\documentclass{amsart}
\usepackage{graphicx}
\usepackage{amssymb}
\usepackage{amsmath}
\usepackage{latexsym}
\usepackage{amsthm}
\usepackage{autograph,epic,latexsym,bezier,amsbsy,color,enumerate,amsmath,amscd,amssymb}
\usepackage[latin1]{inputenc}
\unitlength=0,4mm \setloopdiam{10}\setprofcurve{7}
\newtheorem{defi}{Definition}[section]
\newtheorem{teo}[defi]{Theorem}

\newtheorem{prop}[defi]{Proposition}
\newtheorem{cor}[defi]{Corollary}

\newtheorem{os}[defi]{Remark}

\begin{document}
\title[Ising model on Schreier graphs]{Partition functions of the Ising model on some self-similar Schreier graphs}

\author{Daniele D'Angeli}
\address{Daniele D'Angeli. Department of Mathematics, Technion--Israel Institute of
Technology, Technion City, Haifa 32 000.}
\email{dangeli@tx.technion.ac.il}

\author{Alfredo Donno}
\address{Alfredo Donno. Dipartimento di Matematica ``G. Castelnuovo", Sapienza Università di Roma, Piazzale A. Moro, 2, 00185 Roma, Italia.}
\email{donno@mat.uniroma1.it}

\author{Tatiana Nagnibeda}
\address{Tatiana Nagnibeda. Section de Mathématiques, Université de Genève, 2-4, Rue du Lièvre, Case Postale 64, 1211 Genève 4, Suisse}
\email{Tatiana.Smirnova-Nagnibeda@unige.ch}


\date{\today}
\begin{abstract}
We study partition functions and thermodynamic limits for the
Ising model on three families of finite graphs converging to
infinite self-similar graphs. They are provided by three
well-known groups realized as automorphism groups of regular
rooted trees: the first Grigorchuk's group of intermediate growth;
the iterated monodromy group of the complex polynomial $z^2-1$
known as the \lq\lq Basilica group\rq\rq ; and the Hanoi Towers
group $H^{(3)}$ closely related to the Sierpi\'nsky gasket.
\end{abstract}
\maketitle
\begin{center}
{\footnotesize{\bf Mathematics Subject Classification (2010):}
Primary 82B20; Secondary 05A15.\footnote{This research has been
supported by the Swiss National Science Foundation Grant
PP0022$_{-}$118946.}}
\end{center}

\section{Introduction}
\subsection{The Ising model}\label{intro1.1}
The famous Ising model of ferromagnetism was introduced by W. Lenz
in 1920, \cite{lenz}, and became the subject of the PhD thesis of
his student E. Ising. It consists of discrete variables called
spins arranged on the vertices of a finite graph $Y$. Each spin
can take values ${\pm 1}$ and  only interacts with its nearest
neighbours. Configuration of spins at two adjacent vertices $i$
and $j$ has energy $J_{i,j}>0$ if the spins have opposite values,
and $-J_{i,j}$ if the values are the same. Let $|Vert(Y)|=N$, and
let $\vec\sigma = (\sigma_1,...,\sigma_N)$  denote the
configuration of spins,  with $\sigma_i\in\{\pm 1\}$. The total
energy of the system in configuration $\vec\sigma$ is then
$$
E(\vec\sigma) =-\sum_{i\sim j}J_{i,j}\sigma_i \sigma_j ,
$$
where we write $i \sim j$ if the vertices $i$ and $j$ are adjacent
in $Y$.

The probability of a particular configuration at temperature $T$
is given by
$$ \mathbb{P}(\vec\sigma) = \frac{1}{Z} \exp (-\beta E(\vec\sigma)) ,$$
where $\beta$ is the "inverse temperature" conventionally defined
as $\beta\equiv1/(k_B T)$, and $k_B$ denotes the Boltzmann
constant.

As usual in statistical physics, the normalizing constant that
makes the distribution above a probability measure is called the
{\it partition function}:
$$
Z=\sum_{\vec\sigma}\exp(-\beta E(\vec\sigma)).
$$
One can rewrite this formula by using $\exp(K\sigma_i \sigma_j) =
\cosh(K) + \sigma_i\sigma_j\sinh (K)$, so as to get the so-called
\lq\lq high temperature expansion\rq\rq\ :
\begin{eqnarray*}
Z &=& \prod_{i\sim j}\cosh (\beta J_{i,j}) \sum_{\vec\sigma} (1+  \sum_{i\sim j} \sigma_i \sigma_j \tanh(\beta J_{i,j})\\
&+& \sum_{\begin{array}{c}\scriptstyle i\sim j \\ \scriptstyle
l\sim m  \end{array}}(\sigma_i\sigma_j)(\sigma_l \sigma_m)
\tanh(\beta J_{i,j}) \tanh(\beta J_{l,m})+ \cdots ).
\end{eqnarray*}
After changing the order of summation, observe that the
non-vanishing terms in $Z$ are exactly those with an even number
of occurrences of each $\sigma_i$. We can interpret this by saying
that non-vanishing terms in this expression are in bijection with
{\it closed polygons} of $Y$, i.e.,  subgraphs in which every
vertex has even degree. Consequently we can rewrite $Z$ as
\begin{eqnarray}\label{polygons}
Z=\big(\prod_{i\sim j}\cosh (\beta J_{i,j})\big)\cdot 2^N\sum_{X
\text{ closed polygon of } Y}\prod_{(i,j)\in Edges(X)} \tanh(\beta
J_{i,j})\ ,
\end{eqnarray}
where in the RHS we have the {\it generating series of closed
polygons} of $Y$ with weighted edges, the weight of an edge
$(i,j)$ being $\tanh(\beta J_{i,j})$.

In the case of constant $J$, the above expression specializes to
$$
Z=\big(\cosh(\beta J)\big)^{|Edges(Y)|}\cdot 2^N \Gamma^{cl}(\tanh(\beta J))
$$
with $\Gamma^{cl}(z)=\sum_{n=0}^\infty A_n^{cl} z^n,$ where
$A_n^{cl}$ is the number of closed polygons with $n$ edges in $Y$.
(In particular, the total number of closed polygons is given by
$\Gamma^{cl}(1)$.)

From the physics viewpoint it is interesting to study the model
when the system (i.e. the number of vertices in the graph) grows.
One way to express this mathematically is to consider growing
sequences of finite graphs converging to an infinite graph. If the
limit
$$
\lim_{n\to \infty} \frac{\log(Z_n)}{|Vert(Y_n)|}
$$
for a sequence of finite graphs $Y_n$ with partition functions
$Z_n$ exists, it is called the {\it thermodynamic limit}.

In the thermodynamic limit, at some critical temperature, a phase
transition can occur between ordered and disordered phase in the
behaviour of the model. Existence of a phase transition depends on
the graph. In his thesis in 1925, Ising studied the case of
one-dimensional Euclidean lattice; computed the partition
functions and the thermodynamic limit and showed that there is no
phase transition \cite{ising}. The Ising model in $\mathbb Z^d$
with $d\geq 2$ undergoes a phase transition. This was first
established for $d=2$ by R. Peierls. At high temperature, $T>T_C$,
the clusters of vertices with equal spins grow similarly for two
different types of spin, whereas for $T<T_C$ the densities of the
types of spin are different and the system ``chooses" one of them.

The infinite graphs that will be studied in this paper all have
{\it finite order} $\mathcal R$ {\it of ramification}, i.e., for
any connected bounded part $\mathcal X$ of the graph there exists
a set $\mathcal A$ of at most $\mathcal R$ vertices such that any
infinite self-avoiding path in the graph that begins in $\mathcal
X$ necessarily goes through $\mathcal A$. Finite order of
ramification ensures that the critical temperature is T=0, and
there is no phase transition in the Ising model (see \cite{Mand}).

Typically, an infinite lattice is viewed as the limit of an
exhaustive sequence of finite subgraphs. This is a simple example
of the so-called {\it pointed Hausdorff-Gromov convergence} (see
Proposition \ref{GrH} below). Another typical case of this
convergence is that of covering graph sequences. In this paper we
will be studying the Ising model on families of finite graphs
coming from the theory of {\it self-similar groups} (see
Definition \ref{defiselfsimilar} below), and their infinite
limits. Any finitely generated group of automorphisms of a regular
rooted tree provides us with a sequence of finite graphs
describing the action of the group on the levels of the tree. When
the action is self-similar the sequence converges in the above
sense to infinite graphs describing the action of the group on the
boundary of the tree. The graphs that we study here are determined
by group actions, and so their edges are labeled naturally by the
generators of the acting group. Different weights on the edges
lead to weighted partition functions, with $J_{i,j}$ depending on
the label of the edge $(i,j)$.

\subsection{Groups of automorphisms of rooted regular trees}\label{preliminarigruppi}
Let $T$ be the infinite regular rooted tree of degree $q$, i.e.,
the rooted tree in which each vertex has $q$ children. Each vertex
of the $n$-th level of the tree can be regarded as a word of
length $n$ in the alphabet $X=\{0,1,\ldots, q-1\}$. Moreover, one
can identify the set $X^{\omega}$ of infinite words in $X$ with
the set $\partial T$ of infinite geodesic rays starting at the
root of $T$. Now let $G<Aut(T)$ be a group acting on $T$ by
automorphisms, generated by a finite symmetric set of generators
$S$. Suppose moreover that the action is transitive on each level
of the tree.
\begin{defi}\label{defischreiernovembre}
The $n$-th {\it Schreier graph} $\Sigma_n$ of the action of $G$ on
$T$, with respect to the generating set $S$, is a graph whose
vertex set coincides with the set of vertices of the $n$-th level
of the tree, and two vertices $u,v$ are adjacent if and only if
there exists $s\in S$ such that $s(u)=v$. If this is the case, the
edge joining $u$ and $v$ is labeled by $s$. For any infinite ray
$\xi\in\partial T$, the {\it orbital Schreier graph} $\Sigma_\xi$
has vertices $G\cdot \xi$ and edges determined by the action of
generators, as above.
\end{defi}
The vertices of $\Sigma_n$ are labeled by words of length $n$ in
$X$ and the edges are labeled by elements of $S$. The Schreier
graph is thus a regular graph of degree $d = |S|$ with $q^n$
vertices, and it is connected since the action of $G$ is
level-transitive.
\begin{defi}\label{defiselfsimilar}
A finitely generated group $G<Aut(T)$ is {\it self-similar} if,
for all $g\in G, x\in X$, there exist $h\in G, y\in X$ such that
$$
g(xw)=yh(w),
$$
for all finite words $w$ in the alphabet $X$.
\end{defi}
Self-similarity implies that $G$ can be embedded into the wreath
product $Sym(q)\wr G$, where $Sym(q)$ denotes the symmetric group
on $q$ elements, so that any automorphism $g\in G$ can be
represented as
$$
g=\tau(g_0,\ldots,g_{q-1}),
$$
where $\tau\in Sym(q)$ describes the action of $g$ on the first
level of $T$ and $g_i\in G, i=0,...,q-1$ is the restriction of $g$
on the full subtree of $T$ rooted at the vertex $i$ of the first
level of $T$ (observe that any such subtree is isomorphic to $T$).
Hence, if $x\in X$ and $w$ is a finite word in $X$, we have
$$
g(xw)=\tau(x)g_x(w).
$$
It is not difficult to see that the orbital Schreier graphs of a
self-similar group are infinite and that the finite Schreier
graphs $\{\Sigma_n\}_{n=1}^\infty$ form a sequence of graph
coverings (see \cite{volo} and references therein for more
information about this interesting class of groups, also known as
{\it automata groups}.)

Take now an infinite ray $\xi\in X^\omega$ and denote by $\xi_n$
the $n$-th prefix of the word $\xi$. Then the sequence of rooted
graphs $\{(\Sigma_n,\xi_n)\}$ converges to the infinite rooted
graph $(\Sigma_\xi, \xi)$ in the space of rooted graphs, in the
following sense.
\begin{prop}\label{GrH} (\cite{Grom}, Chapter 3.) Let $\mathcal{X}$
be the space of connected graphs having a distinguished vertex
called {\it the root}; $\mathcal{X}$ can be endowed with the
following metric: given two rooted graphs $(Y_1,v_1)$ and
$(Y_2,v_2)$,
$$
Dist((Y_1,v_1),(Y_2,v_2)):=\inf\left\{\frac{1}{r+1};\textrm{$B_{Y_1}(v_1,r)$
is isomorphic to $B_{Y_2}(v_2,r)$}\right\}
$$
where $B_Y(v,r)$ is the ball of radius $r$ in $Y$ centered in $v$.
Under the assumption of uniformly bounded degrees, $\mathcal{X}$
endowed with the metric $Dist$ is a compact space.
\end{prop}

\subsection{Plan of the paper}
\indent Our aim in this paper is to study the Ising model on the
Schreier graphs of three key examples of self-similar groups:\\
\indent- the first Grigorchuk's  group of intermediate (i.e.,
strictly between polynomial and exponential) growth (see
\cite{grigorchuk} for a detailed account and further
references);\\ \indent- the \lq\lq Basilica\rq\rq group that can
be described as the iterated monodromy group of the complex
polynomial $z^2-1$ (see \cite{volo} connections of self-similar
groups to complex dynamics);\\ \indent- and the Hanoi Towers group
$H^{(3)}$ whose action on the ternary tree models the famous Hanoi
Towers game on three pegs, see \cite{hanoi}.

It is known \cite{bondarenkothesis} that the infinite Schreier
graphs associated with these groups (and, more generally, with all
groups generated by {\it bounded automata}) have  finite order of
ramification. Hence the Ising model on these graphs exhibits no
phase transition.

We first compute the partition functions and prove existence of
thermodynamic limit for the model where interactions between
vertices are constant: in Section \ref{chaptergrigorchuk} we treat
the Grigorchuk's group and the Basilica group, and in Section
\ref{chapterHanoi} the Hanoi Towers group $H^{(3)}$ and its close
relative the Sierpi\'{n}ski gasket are considered.

In Section \ref{chapterstatitics}, we study weighted partition
functions for all the graphs previously considered, and we find
the distribution of the number of occurrences of a fixed weight in
a random configuration. The relation between the Schreier graphs
of $H^{(3)}$ and the Sierpi\'{n}ski gasket is also discussed from
the viewpoint of Fisher's theorem establishing a correspondence
between the Ising model on the Sierpi\'{n}ski gasket and the
dimers model on the Schreier graphs of $H^{(3)}$.

\subsection{Acknowledgments} We are grateful to R. Grigorchuk for numerous inspiring discussions and to S. Smirnov for useful remarks on the first version of this paper.


\section{Partition functions and thermodynamic limit for the Grigorchuk's
group and for the Basilica group}\label{chaptergrigorchuk}

\subsection{Grigorchuk's group}\label{gri}
This group admits the following easy description as a self-similar
subgroup of automorphisms of the binary tree. It is generated by
the elements
$$
a=\epsilon(id,id), \ \ \ b=e(a,c), \ \ c=e(a,d), \ \ d=e(id,b),
$$
where $e$ and $\epsilon$ are respectively the trivial and the
non-trivial permutations in $Sym(2)$. These recursive formulae
allow easily to construct finite Schreier graphs for the action of
the group on the binary tree. Here are three first graphs in the
sequence, with loops erased.
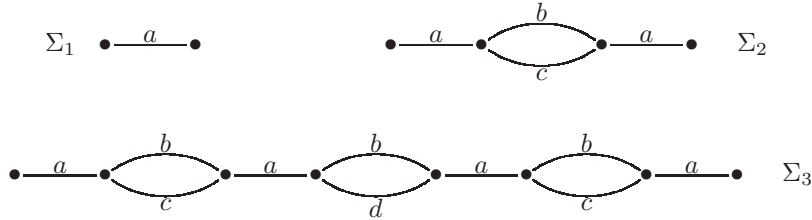
\begin{center}
\begin{picture}(300,40)

\letvertex A=(60,20)\letvertex B=(90,20)

\letvertex C=(155,20)\letvertex D=(185,20)

\letvertex E=(225,20)\letvertex F=(255,20)

\drawundirectededge(A,B){$a$} \drawundirectededge(C,D){$a$}

\drawundirectededge(E,F){$a$}

\drawundirectedcurvededge(D,E){$b$}\drawundirectedcurvededge(E,D){$c$}

\put(40,18){$\Sigma_1$} \put(270,18){$\Sigma_2$}

\drawvertex(A){$\bullet$}\drawvertex(B){$\bullet$}

\drawvertex(C){$\bullet$}\drawvertex(D){$\bullet$}

\drawvertex(E){$\bullet$}\drawvertex(F){$\bullet$}

\end{picture}

\end{center}
\begin{center}

\begin{picture}(300,40)

\letvertex A=(30,20)\letvertex B=(60,20)\letvertex C=(100,20)\letvertex D=(130,20)

\letvertex E=(170,20)\letvertex F=(200,20)\letvertex G=(240,20)\letvertex H=(270,20)

\drawundirectededge(A,B){$a$} \drawundirectededge(C,D){$a$}

\drawundirectededge(E,F){$a$} \drawundirectededge(G,H){$a$}

\put(285,18){$\Sigma_3$}

\drawvertex(A){$\bullet$}\drawvertex(B){$\bullet$}

\drawvertex(C){$\bullet$}\drawvertex(D){$\bullet$}

\drawvertex(E){$\bullet$}\drawvertex(F){$\bullet$}

\drawvertex(G){$\bullet$}\drawvertex(H){$\bullet$}

\drawundirectedcurvededge(B,C){$b$}

\drawundirectedcurvededge(C,B){$c$}

\drawundirectedcurvededge(D,E){$b$}

\drawundirectedcurvededge(E,D){$d$}

\drawundirectedcurvededge(F,G){$b$}

\drawundirectedcurvededge(G,F){$c$}
\end{picture}
\end{center}
In general, the Schreier graph $\Sigma_n$ has the same linear
shape, with $2^{n-1}$ simple edges, all labeled by $a$, and
$2^{n-1}-1$ cycles of length 2. It is therefore very easy to
compute the generating function of closed polygons of $\Sigma_n$,
for each $n$.
\begin{teo}
The generating function of closed polygons for the $n$-th Schreier
graph of the Grigorchuk group is $ \Gamma^{cl}_n(z)=
(1+z^2)^{2^{n-1}-1}$. In particular, the number of all closed
polygons in $\Sigma_n$ is $2^{2^{n-1}-1}$.

The partition function of the Ising model is given by
$$
Z_n=\cosh(\beta J)^{3\cdot 2^{n-1}-2}\cdot 2^{2^n}\cdot
\left(1+\tanh^2\left(\beta J\right)\right)^{2^{n-1}-1}
$$
and the thermodynamic limit exists and satisfies:
$$
\lim_{n\to \infty}\frac{\log(Z_n)}{2^n} =
\frac{3}{2}\log(\cosh(\beta J)) +
\log2+\frac{1}{2}\log\left(1+\tanh^2\left(\beta J\right)\right).
$$
\end{teo}
\begin{proof}
Is is clear that a closed polygon in $\Sigma_n$ is the union of
2-cycles. So we can easily compute the number $A^{cl}_{k,n}$ of
closed polygons with $k$ edges in $\Sigma_n$, for all
$k=0,1,\ldots, 2^n-2$. For $k$ odd, one has $A_{k,n}^{cl}=0$. For
$k$ even, we have to choose $\frac{k}{2}$ cycles of length $2$ to
get a closed polygon with $k$ edges, which implies
$$
A_{k,n}^{cl}= {2^{n-1}-1 \choose  \frac{k}{2}}.
$$
So the generating function of closed polygons for $\Sigma_n$ is
given by
$$
\Gamma^{cl}_n(z)= \sum_{k=0}^{2^{n-1}-1} {2^{n-1}-1 \choose
k}z^{2k} = (1+z^2)^{2^{n-1}-1}.
$$
\end{proof}


\subsection{The  Basilica group}\label{chapterBasilica}
The Basilica group is a self-similar group of automorphisms of the
binary tree generated by the elements
$$
a=e(b,id), \ \ \ \ \ b=\epsilon(a,id).
$$
The following pictures of graphs $\Sigma_n$ for $n=1,2,3,4,5$ with
loops erased give an idea of how finite Schreier graphs of the
Basilica group look like. See \cite{schreierbasilica} for a
comprehensive analysis of finite and infinite Schreier graphs of
this group. Note also that $\{\Sigma_n\}_{n=1}^\infty$ is an
approximating sequence for the Julia set of the polynomial
$z^2-1$, the famous \lq\lq Basilica\rq\rq fractal (see
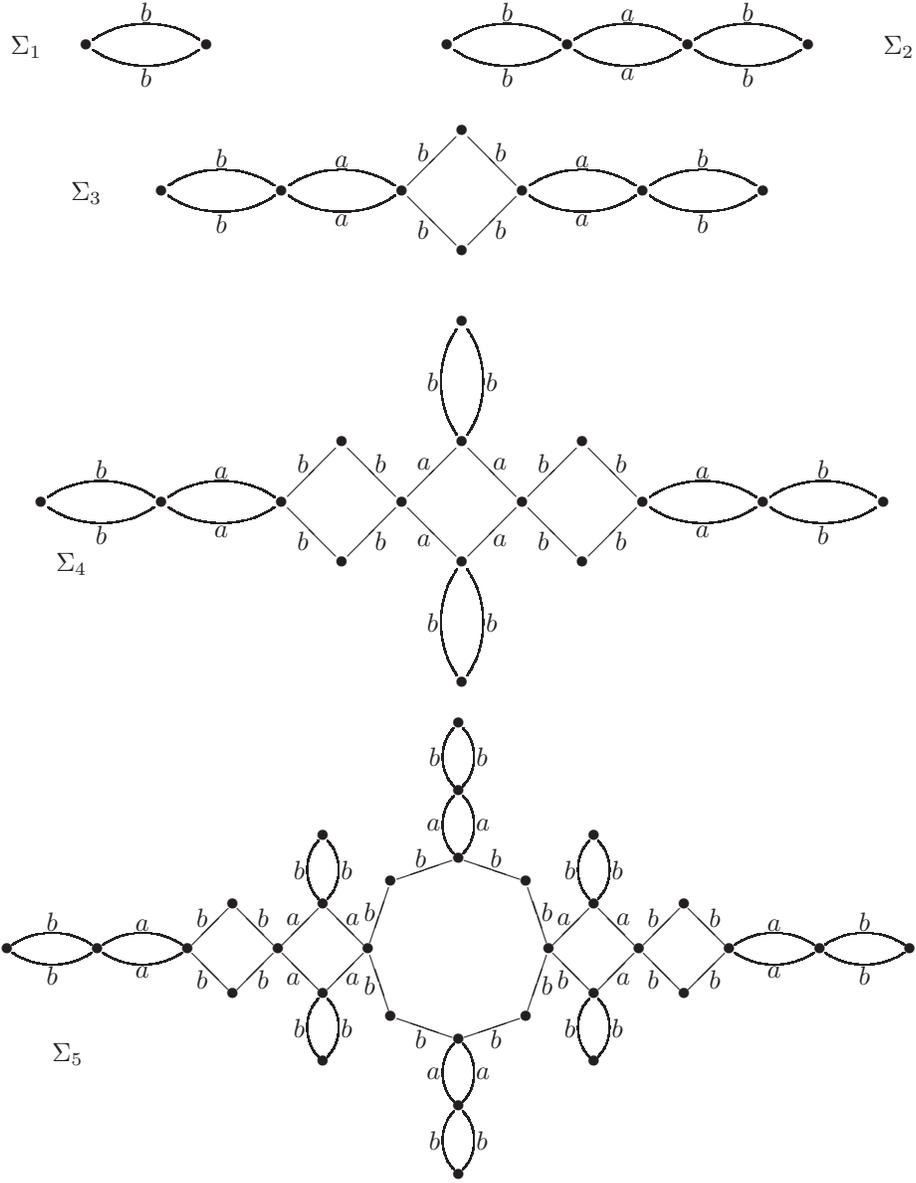
\cite{volo}).
\begin{center}
\begin{picture}(300,30)

\letvertex A=(30,15)\letvertex B=(70,15)\letvertex C=(150,15)\letvertex D=(190,15)

\letvertex E=(230,15)\letvertex F=(270,15)

\drawvertex(A){$\bullet$}\drawvertex(B){$\bullet$}

\drawvertex(C){$\bullet$}\drawvertex(D){$\bullet$}

\drawvertex(E){$\bullet$}\drawvertex(F){$\bullet$}

\drawundirectedcurvededge(A,B){$b$}

\drawundirectedcurvededge(B,A){$b$}

\drawundirectedcurvededge(C,D){$b$}

\drawundirectedcurvededge(D,C){$b$}

\drawundirectedcurvededge(D,E){$a$}

\drawundirectedcurvededge(E,D){$a$}

\drawundirectedcurvededge(E,F){$b$}

\drawundirectedcurvededge(F,E){$b$} \put(5,12){$\Sigma_1$}

\put(295,12){$\Sigma_2$}

\end{picture}
\end{center}
\begin{center}
\begin{picture}(300,60)

\letvertex A=(55,30)\letvertex B=(95,30)\letvertex C=(135,30)\letvertex D=(155,50)

\letvertex E=(155,10)\letvertex F=(175,30)\letvertex G=(215,30)\letvertex H=(255,30)

\drawvertex(A){$\bullet$}\drawvertex(B){$\bullet$}

\drawvertex(C){$\bullet$}\drawvertex(D){$\bullet$}

\drawvertex(E){$\bullet$}\drawvertex(F){$\bullet$}

\drawvertex(G){$\bullet$}\drawvertex(H){$\bullet$}

\drawundirectededge(C,D){$b$} \drawundirectededge(E,C){$b$}

\drawundirectededge(F,E){$b$} \drawundirectededge(D,F){$b$}

\drawundirectedcurvededge(A,B){$b$}

\drawundirectedcurvededge(B,A){$b$}

\drawundirectedcurvededge(B,C){$a$}

\drawundirectedcurvededge(C,B){$a$}

\drawundirectedcurvededge(F,G){$a$}

\drawundirectedcurvededge(G,F){$a$}

\drawundirectedcurvededge(G,H){$b$}

\drawundirectedcurvededge(H,G){$b$} \put(25,27){$\Sigma_3$}
\end{picture}
\end{center}
\begin{center}
\begin{picture}(300,140)

\letvertex A=(15,70)

\letvertex B=(55,70)\letvertex C=(95,70)\letvertex D=(115,90)\letvertex E=(115,50)

\letvertex F=(135,70)\letvertex G=(155,90)\letvertex H=(155,50)\letvertex I=(175,70)

\letvertex J=(155,130)\letvertex K=(155,10)

\letvertex L=(195,90)

\letvertex M=(195,50)\letvertex N=(215,70)

\letvertex O=(255,70)\letvertex P=(295,70)

\drawvertex(A){$\bullet$}\drawvertex(B){$\bullet$}

\drawvertex(C){$\bullet$}\drawvertex(D){$\bullet$}

\drawvertex(E){$\bullet$}\drawvertex(F){$\bullet$}

\drawvertex(G){$\bullet$}\drawvertex(H){$\bullet$}

\drawvertex(I){$\bullet$}\drawvertex(L){$\bullet$}

\drawvertex(M){$\bullet$}\drawvertex(N){$\bullet$}

\drawvertex(O){$\bullet$}\drawvertex(P){$\bullet$}

\drawvertex(J){$\bullet$}\drawvertex(K){$\bullet$}

\drawundirectedcurvededge(A,B){$b$}\drawundirectedcurvededge(B,A){$b$}

\drawundirectedcurvededge(B,C){$a$}\drawundirectedcurvededge(C,B){$a$}

\drawundirectededge(C,D){$b$} \drawundirectededge(D,F){$b$}

\drawundirectededge(F,E){$b$} \drawundirectededge(E,C){$b$}

\drawundirectededge(F,G){$a$} \drawundirectededge(G,I){$a$}

\drawundirectededge(I,H){$a$} \drawundirectededge(H,F){$a$}

\drawundirectededge(I,L){$b$} \drawundirectededge(L,N){$b$}

\drawundirectededge(N,M){$b$} \drawundirectededge(M,I){$b$}

\drawundirectedcurvededge(G,J){$b$}\drawundirectedcurvededge(J,G){$b$}

\drawundirectedcurvededge(H,K){$b$}\drawundirectedcurvededge(K,H){$b$}

\drawundirectedcurvededge(N,O){$a$}\drawundirectedcurvededge(O,N){$a$}

\drawundirectedcurvededge(O,P){$b$}\drawundirectedcurvededge(P,O){$b$}

\put(20,47){$\Sigma_4$}
\end{picture}
\end{center}
\unitlength=0,3mm
\begin{center}
\begin{picture}(400,210)

\letvertex A=(5,110)\letvertex B=(45,110)\letvertex C=(85,110)\letvertex D=(105,130)

\letvertex E=(105,90)\letvertex F=(125,110)\letvertex G=(145,130)\letvertex H=(145,160)

\letvertex I=(165,110)\letvertex L=(145,90)\letvertex M=(145,60)\letvertex N=(175,140)

\letvertex O=(205,150)\letvertex R=(235,140)\letvertex S=(245,110)\letvertex T=(235,80)

\letvertex U=(205,70)\letvertex V=(175,80)\letvertex P=(205,180)\letvertex Q=(205,210)

\letvertex Z=(205,40)\letvertex J=(205,10)\letvertex K=(265,130)\letvertex X=(285,110)

\letvertex W=(265,90)\letvertex g=(265,160)\letvertex h=(265,60)\letvertex c=(305,130)

\letvertex Y=(305,90)\letvertex d=(325,110)\letvertex e=(365,110)\letvertex f=(405,110)

\drawvertex(A){$\bullet$}\drawvertex(B){$\bullet$}

\drawvertex(C){$\bullet$}\drawvertex(D){$\bullet$}

\drawvertex(E){$\bullet$}\drawvertex(F){$\bullet$}

\drawvertex(G){$\bullet$}\drawvertex(H){$\bullet$}

\drawvertex(I){$\bullet$}\drawvertex(L){$\bullet$}

\drawvertex(M){$\bullet$}\drawvertex(N){$\bullet$}

\drawvertex(O){$\bullet$}\drawvertex(P){$\bullet$}

\drawvertex(J){$\bullet$}\drawvertex(K){$\bullet$}

\drawvertex(Q){$\bullet$}\drawvertex(R){$\bullet$}

\drawvertex(S){$\bullet$}\drawvertex(T){$\bullet$}

\drawvertex(U){$\bullet$}\drawvertex(V){$\bullet$}

\drawvertex(W){$\bullet$}\drawvertex(X){$\bullet$}

\drawvertex(Y){$\bullet$}\drawvertex(Z){$\bullet$}

\drawvertex(g){$\bullet$}\drawvertex(h){$\bullet$}

\drawvertex(c){$\bullet$}\drawvertex(f){$\bullet$}

\drawvertex(d){$\bullet$}\drawvertex(e){$\bullet$}

\drawundirectedcurvededge(A,B){$b$}\drawundirectedcurvededge(B,A){$b$}

\drawundirectedcurvededge(B,C){$a$}\drawundirectedcurvededge(C,B){$a$}

\drawundirectededge(C,D){$b$} \drawundirectededge(D,F){$b$}

\drawundirectededge(F,E){$b$} \drawundirectededge(E,C){$b$}

\drawundirectededge(F,G){$a$} \drawundirectededge(G,I){$a$}

\drawundirectededge(I,L){$a$} \drawundirectededge(L,F){$a$}

\drawundirectedcurvededge(G,H){$b$}\drawundirectedcurvededge(H,G){$b$}

\drawundirectedcurvededge(L,M){$b$}\drawundirectedcurvededge(M,L){$b$}

\drawundirectededge(I,N){$b$} \drawundirectededge(N,O){$b$}

\drawundirectededge(O,R){$b$} \drawundirectededge(R,S){$b$}

\drawundirectededge(S,T){$b$} \drawundirectededge(T,U){$b$}

\drawundirectededge(U,V){$b$} \drawundirectededge(V,I){$b$}

\drawundirectedcurvededge(O,P){$a$}\drawundirectedcurvededge(P,O){$a$}

\drawundirectedcurvededge(Q,P){$b$}\drawundirectedcurvededge(P,Q){$b$}

\drawundirectedcurvededge(U,Z){$a$}\drawundirectedcurvededge(Z,U){$a$}

\drawundirectedcurvededge(Z,J){$b$}\drawundirectedcurvededge(J,Z){$b$}

\drawundirectededge(S,K){$a$} \drawundirectededge(K,X){$a$}

\drawundirectededge(X,W){$a$} \drawundirectededge(W,S){$b$}

\drawundirectededge(X,c){$b$} \drawundirectededge(c,d){$b$}

\drawundirectededge(d,Y){$b$} \drawundirectededge(Y,X){$b$}

\drawundirectedcurvededge(d,e){$a$}\drawundirectedcurvededge(e,d){$a$}

\drawundirectedcurvededge(e,f){$b$}\drawundirectedcurvededge(f,e){$b$}

\drawundirectedcurvededge(K,g){$b$}\drawundirectedcurvededge(g,K){$b$}

\drawundirectedcurvededge(W,h){$b$}\drawundirectedcurvededge(h,W){$b$}

\put(25,60){$\Sigma_5$}
\end{picture}
\end{center}
In general, it follows from the recursive definition of the
generators, that each $\Sigma_n$ is a cactus, i.e., a union of
cycles (in this example all of them are of length power of 2)
arranged in a tree-like way. The maximal length of a cycle in
$\Sigma_n$ is $2^{\lceil\frac{n}{2}\rceil}$. Denote by $a^i_j$ the
number of cycles of length $j$ labeled by $a$ in $\Sigma_i$ and
analogously denote by $b^i_j$ the number of cycles of length $j$
labeled by $b$ in $\Sigma_i$.
\begin{prop}\label{bascycles}
For any $n\geq 4$ consider the Schreier graph $\Sigma_n$ of the
Basilica group. For each $k\geq 1$, the number of cycles of length
$2^k$ labeled by $a$ is
$$
a^n_{2^k} =
\begin{cases}
2^{n-2k-1} & \text{for }1\leq k \leq \frac{n-1}{2}-1\\
2 & \text{for } k = \lfloor\frac{n}{2}\rfloor
\end{cases}, \qquad \mbox{for } n \mbox{ odd},
$$
$$
a^n_{2^k}= \begin{cases}
2^{n-2k-1} & \text{for }1\leq k \leq \frac{n}{2}-1\\
1 & \text{for } k = \frac{n}{2}
\end{cases}, \qquad \mbox{for }  n \mbox{ even}
$$
and the number of cycles of length $2^k$ labeled by $b$ is
$$
b^n_{2^k}=\begin{cases}
2^{n-2k} & \text{for }1\leq k \leq \frac{n-1}{2}-1\\
2 & \text{for } k = \lfloor\frac{n}{2}\rfloor\\
1 & \text{for } k = \lceil\frac{n}{2}\rceil
\end{cases}, \qquad \mbox{for } n \mbox{ odd},
$$
$$
b^n_{2^k}=\begin{cases}
2^{n-2k} & \text{for }1\leq k \leq \frac{n}{2}-1\\
2& \text{for } k = \frac{n}{2}
\end{cases}, \qquad \mbox{for }  n \mbox{ even}.
$$
\end{prop}
\begin{proof}
The recursive formulae for the generators imply that, for each
$n\geq 3$, one has
$$
a^n_2 = b^{n-1}_2 \ \ \ \mbox{and } \ \ \ b^n_2 = a^{n-1}_1 =
2^{n-2}
$$
and in general $a^n_{2^k} = a^{n-2(k-1)}_2$ and $b^n_{2^k} =
b^{n-2(k-1)}_2$. In particular, for each $n\geq 4$, the number of
$2$-cycles labeled by $a$ is $2^{n-3}$ and the number of
$2$-cycles labeled by $b$ is $2^{n-2}$. More generally, the number
of cycles of length $2^k$ is given by
$$
a^n_{2^k} = 2^{n-2k-1}, \ \ \ \ b^n_{2^k}=2^{n-2k},
$$
where the last equality is true if $n-2k+2 \geq 4$, i.e. for
$k\leq \frac{n}{2}-1$. Finally, for $n$ odd, there is only one
cycle of length $2^{\lceil\frac{n}{2}\rceil}$ labeled by $b$ and
four cycles of length $2^{\lfloor\frac{n}{2}\rfloor}$, two of them
labeled by $a$ and two labeled by $b$; for $n$ even, there are
three cycles of length $2^{\frac{n}{2}}$, two of them labeled by
$b$ and one labeled by $a$.
\end{proof}
\begin{cor}
For each $n\geq 4$, the number of cycles labeled by $a$ in the
Schreier graph $\Sigma_n$ of the Basilica group is
$$
\begin{cases}
\frac{2^{n-1}+2}{3}  & \ \text{for}\ n \ \text{odd},\\
\frac{2^{n-1}+1}{3}  & \ \text{for}\ n \ \text{even}.
\end{cases}
$$
and the number of $b$-cycles in $\Sigma_n$ is
$$
\begin{cases}
\frac{2^n+1}{3}  & \ \text{for}\  n \ \text{odd},\\
\frac{2^n+2}{3}  & \ \text{for}\  n \ \text{even}.
\end{cases}
$$
The total number of cycles of length $\geq 2$ is $2^{n-1}+1$ and
the total number of edges, without loops, is $3\cdot 2^{n-1}$.
\end{cor}


The computations above lead to the following formula for the
partition function of the Ising model on the Schreier graphs
$\Sigma_n$ associated with the action of the Basilica group.
\begin{teo}\label{basilicapartfunct}
The partition function of the Ising model on the $n$-th Schreier
graph $\Sigma_n$ of the Basilica group is
$$
Z_n=2^{2^n}\cdot \cosh(\beta J)^{3\cdot
2^{n-1}}\cdot\Gamma^{cl}_n(\tanh(\beta J)),
$$
where $\Gamma_n^{cl}(z)$ is the generating function of closed
polygons for $\Sigma_n$ given by
$$
\Gamma^{cl}_n(z) =
\prod_{k=1}^{\frac{n-1}{2}-1}\left(1+z^{2^k}\right)^{3\cdot
2^{n-2k-1}}\cdot \left(1+z^{2^{\frac{n-1}{2}}}\right)^4\cdot
\left(1+z^{2^{\frac{n+1}{2}}}\right),
$$
for $n\geq 5$ odd and
$$
\Gamma^{cl}_n(z) =
\prod_{k=1}^{\frac{n}{2}-1}\left(1+z^{2^k}\right)^{3\cdot
2^{n-2k-1}}\cdot \left(1+z^{2^{\frac{n}{2}}}\right)^3,
$$
for $n\geq 4$ even. Moreover, $\Gamma_1^{cl}=1+z^2$,
$\Gamma_2^{cl}=(1+z^2)^3$ and $\Gamma_3^{cl}=(1+z^2)^4(1+z^4)$.
\end{teo}
\begin{proof}
Recall that $Z_n = 2^{|Vert(\Sigma_n)|}\cosh(\beta
J)^{|Edges(\Sigma_n)|}\cdot\Gamma_n^{cl}(\tanh(\beta J))$, where
$\Gamma_n^{cl}(z)$ is the generating function of closed polygons
in $\Sigma_n$. In our case we have $|Edges(\Sigma_n)| \\= 3\cdot
2^{n-1}$ and $|Vert(\Sigma_n)|= 2^n$.

The formulae for $\Gamma_n^{cl}(z)$ with $n=1,2,3$ can be directly
verified. For $n\geq 4$, we can use Proposition \ref{bascycles}.
Since the length of each cycle of $\Sigma_n$ is even, it is clear
that the coefficient $A^{cl}_{k,n}$ is zero for every odd $k$. The
coefficient $A^{cl}_{k,n}$ is nonzero for every even $k$ such that
$0 \leq k \leq 3\cdot 2^{n-1}$. In fact, $3\cdot 2^{n-1}$ is the
total number of edges of $\Sigma_n$ ($2^n$ labeled by $b$ and
$2^{n-1}$ labeled by $a$). By taking the exact number of cycles of
length $2^i$ in $\Sigma_n$, we get the assertion.
\end{proof}
\begin{teo}\label{EXISTENCEbasilica}
The thermodynamic limit $ \lim_{n\to
\infty}\frac{\log(Z_n)}{|Vert(\Sigma_n)|}$ exists.
\end{teo}
\begin{proof}
Since $|Edges(\Sigma_n)| = 3\cdot 2^{n-1}$ and $|Vert(\Sigma_n)|=
2^n$, the limit reduces to (choosing, for example, $n$ even)
$$
\log(2) + \frac{3}{2}\log(\cosh(\beta J)) +
\lim_{n\to\infty}\frac{\log(\Gamma_n^{cl}(z))}{2^n},
$$
where $z=\tanh (\beta J)$ takes values between 0 and 1. Now
\begin{eqnarray*}
\lim_{n\to \infty}\frac{\log(\Gamma_n^{cl}(z))}{2^n} &=&
\lim_{n\to \infty}\frac{\sum_{k=1}^{\frac{n}{2}-1}3\cdot
2^{n-2k-1}\log(1+z^{2^k})+3\log(1+z^{2^{\frac{n}{2}}})}{2^n}\\ &=&
\frac{3}{2}\sum_{k=1}^{\infty}\frac{\log(1+z^{2^k})}{4^k}+\lim_{n\to\infty}\frac{3\log(1+z^{2^{\frac{n}{2}}})}{2^n}\\&\leq&
\frac{3}{2}\sum_{k=1}^{\infty}\frac{\log(2)}{4^k}<\infty,
\end{eqnarray*}
giving the assertion.
\end{proof}

\section{Partition functions and thermodynamic limits for the Hanoi Towers group $H^{(3)}$ and for the Sierpi\'nski gasket} \label{chapterHanoi}
\subsection{Hanoi Towers group $H^{(3)}$}\label{grafohanoigrafo}

The Hanoi Towers group $H^{(3)}$ is  generated by three
automorphisms of the ternary rooted tree admitting the following
self-similar presentation \cite{hanoi}:
$$
a= (01)(id,id,a) \ \ \ \ b= (02)(id,b,id) \ \ \ \ c=(12)(c,id,id),
$$
where $(01), (02)$ and $(12)$ are transpositions in $Sym(3)$.  The
associated Schreier graphs are self-similar in the sense of
\cite{wagner2}, that is, each $\Sigma_{n+1}$ contains three copies
of $\Sigma_n$  glued together by three edges. These graphs can be
recursively constructed via the following substitutional rules
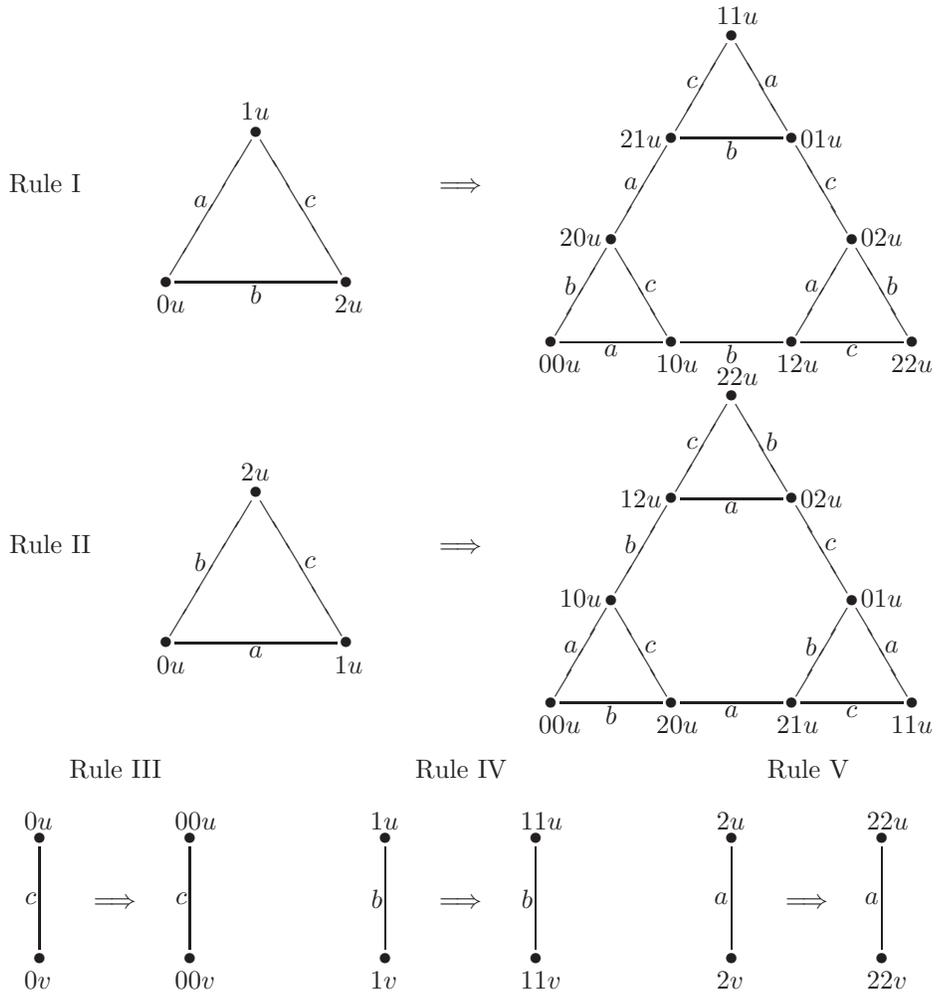
\cite{hanoi}: \unitlength=0,4mm
\begin{center}
\begin{picture}(400,115)
\letvertex A=(190,10)\letvertex B=(210,44)

\letvertex C=(230,78)\letvertex D=(250,112)

\letvertex E=(270,78)\letvertex F=(290,44)

\letvertex G=(310,10)\letvertex H=(270,10)\letvertex I=(230,10)

\letvertex L=(62,30)\letvertex M=(122,30)

\letvertex N=(92,80)

\put(186,0){$00u$}\put(193,42){$20u$}\put(213,75){$21u$}

\put(245,116){$11u$}\put(273,75){$01u$}\put(293,42){$02u$}\put(303,0){$22u$}

\put(265,0){$12u$}\put(225,0){$10u$}

\put(59,20){$0u$}\put(118,20){$2u$}\put(87,84){$1u$}\put(153,60){$\Longrightarrow$}

\put(10,60){Rule I}

\drawvertex(A){$\bullet$}\drawvertex(B){$\bullet$}

\drawvertex(C){$\bullet$}\drawvertex(D){$\bullet$}

\drawvertex(E){$\bullet$}\drawvertex(F){$\bullet$}

\drawvertex(G){$\bullet$}\drawvertex(H){$\bullet$}

\drawvertex(I){$\bullet$}

\drawundirectededge(A,B){$b$}\drawundirectededge(B,C){$a$}\drawundirectededge(C,D){$c$}

\drawundirectededge(D,E){$a$}\drawundirectededge(E,C){$b$}\drawundirectededge(E,F){$c$}\drawundirectededge(F,G){$b$}

\drawundirectededge(B,I){$c$}\drawundirectededge(H,F){$a$}\drawundirectededge(H,I){$b$}

\drawundirectededge(I,A){$a$}\drawundirectededge(G,H){$c$}

\drawvertex(L){$\bullet$}

\drawvertex(M){$\bullet$}\drawvertex(N){$\bullet$}

\drawundirectededge(M,L){$b$}\drawundirectededge(N,M){$c$}\drawundirectededge(L,N){$a$}
\end{picture}
\end{center}
\begin{center}
\begin{picture}(400,105)
\letvertex A=(190,10)\letvertex B=(210,44)

\letvertex C=(230,78)\letvertex D=(250,112)

\letvertex E=(270,78)\letvertex F=(290,44)

\letvertex G=(310,10)\letvertex H=(270,10)\letvertex I=(230,10)

\letvertex L=(62,30)\letvertex M=(122,30)

\letvertex N=(92,80)

\put(186,0){$00u$}\put(193,42){$10u$}\put(213,75){$12u$}

\put(245,116){$22u$}\put(273,75){$02u$}\put(293,42){$01u$}\put(303,0){$11u$}

\put(265,0){$21u$}\put(225,0){$20u$}

\put(59,20){$0u$}\put(118,20){$1u$}\put(87,84){$2u$}\put(153,60){$\Longrightarrow$}

\put(10,60){Rule II}

\drawvertex(A){$\bullet$}\drawvertex(B){$\bullet$}

\drawvertex(C){$\bullet$}\drawvertex(D){$\bullet$}

\drawvertex(E){$\bullet$}\drawvertex(F){$\bullet$}

\drawvertex(G){$\bullet$}\drawvertex(H){$\bullet$}

\drawvertex(I){$\bullet$}

\drawundirectededge(A,B){$a$}\drawundirectededge(B,C){$b$}\drawundirectededge(C,D){$c$}

\drawundirectededge(D,E){$b$}\drawundirectededge(E,C){$a$}\drawundirectededge(E,F){$c$}\drawundirectededge(F,G){$a$}

\drawundirectededge(B,I){$c$}\drawundirectededge(H,F){$b$}\drawundirectededge(H,I){$a$}

\drawundirectededge(I,A){$b$}\drawundirectededge(G,H){$c$}

\drawvertex(L){$\bullet$}

\drawvertex(M){$\bullet$}\drawvertex(N){$\bullet$}

\drawundirectededge(M,L){$a$}\drawundirectededge(N,M){$c$}\drawundirectededge(L,N){$b$}
\end{picture}
\end{center}
\begin{center}
\begin{picture}(400,70)
\letvertex A=(20,10)\letvertex B=(70,10)

\letvertex C=(135,10)\letvertex D=(185,10)

\letvertex E=(250,10)\letvertex F=(300,10)

\letvertex G=(20,50)\letvertex H=(70,50)

\letvertex I=(135,50)\letvertex L=(185,50)

\letvertex M=(250,50)\letvertex N=(300,50)

\put(15,53){$0u$}\put(15,0){$0v$}

\put(65,0){$00v$}\put(65,53){$00u$}\put(130,0){$1v$}\put(130,53){$1u$}\put(180,0){$11v$}

\put(180,53){$11u$}\put(245,0){$2v$}\put(245,53){$2u$}\put(295,0){$22v$}\put(295,53){$22u$}

\put(38,27){$\Longrightarrow$}\put(153,27){$\Longrightarrow$}\put(268,27){$\Longrightarrow$}

\put(30,70){Rule III} \put(145,70){Rule IV} \put(262,70){Rule V}

\drawvertex(A){$\bullet$}\drawvertex(B){$\bullet$}

\drawvertex(C){$\bullet$}\drawvertex(D){$\bullet$}

\drawvertex(E){$\bullet$}\drawvertex(F){$\bullet$}

\drawvertex(G){$\bullet$}\drawvertex(H){$\bullet$}

\drawvertex(I){$\bullet$}\drawvertex(L){$\bullet$}

\drawvertex(M){$\bullet$}\drawvertex(N){$\bullet$}

\drawundirectededge(A,G){$c$}\drawundirectededge(B,H){$c$}\drawundirectededge(C,I){$b$}

\drawundirectededge(D,L){$b$}\drawundirectededge(E,M){$a$}\drawundirectededge(F,N){$a$}
\end{picture}
\end{center}
The starting point is the Schreier graph $\Sigma_1$ of the first
level. \unitlength=0,3mm
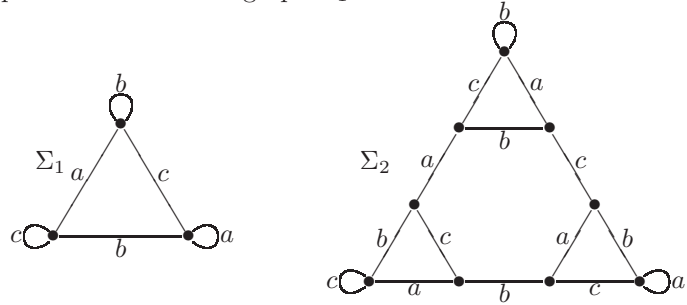
\begin{center}
\begin{picture}(400,128)

\letvertex A=(220,10)\letvertex B=(240,44)

\letvertex C=(260,78)\letvertex D=(280,112)

\letvertex E=(300,78)\letvertex F=(320,44)

\letvertex G=(340,10)\letvertex H=(300,10)\letvertex I=(260,10)

\letvertex L=(80,30)\letvertex M=(140,30)

\letvertex N=(110,80)

\put(216,60){$\Sigma_2$} \put(72,60){$\Sigma_1$}

\drawvertex(A){$\bullet$}\drawvertex(B){$\bullet$}

\drawvertex(C){$\bullet$}\drawvertex(D){$\bullet$}

\drawvertex(E){$\bullet$}\drawvertex(F){$\bullet$}

\drawvertex(G){$\bullet$}\drawvertex(H){$\bullet$}

\drawvertex(I){$\bullet$}

\drawundirectededge(A,B){$b$}\drawundirectededge(B,C){$a$}\drawundirectededge(C,D){$c$}

\drawundirectededge(D,E){$a$}\drawundirectededge(E,C){$b$}\drawundirectededge(E,F){$c$}\drawundirectededge(F,G){$b$}

\drawundirectededge(B,I){$c$}\drawundirectededge(H,F){$a$}\drawundirectededge(H,I){$b$}

\drawundirectededge(I,A){$a$}\drawundirectededge(G,H){$c$}

\drawvertex(L){$\bullet$}

\drawvertex(M){$\bullet$}\drawvertex(N){$\bullet$}

\drawundirectededge(M,L){$b$}\drawundirectededge(N,M){$c$}\drawundirectededge(L,N){$a$}

\drawundirectedloop[l](A){$c$}\drawundirectedloop(D){$b$}\drawundirectedloop[r](G){$a$}\drawundirectedloop[r](M){$a$}

\drawundirectedloop(N){$b$}\drawundirectedloop[l](L){$c$}
\end{picture}
\end{center}
\begin{os}\rm
Observe that, for each $n\geq 1$, the graph $\Sigma_n$ has three
loops, at the vertices $0^n,1^n$ and $2^n$, labeled by $c,b$ and
$a$, respectively. Moreover, these are the only loops in
$\Sigma_n$. The Ising model will be studied on $\Sigma_n$
considered without loops.
\end{os}
Let us now proceed to the computation of closed polygons in
$\Sigma_n$. Denote by $P_n$ the set of closed polygons in
$\Sigma_n$, and by $L_n$ the set of all subgraphs of $\Sigma_n$
consisting of self-avoiding paths joining the left-most vertex to
the right-most vertex in $\Sigma_n$, together with closed polygons
having no common edge with the path.
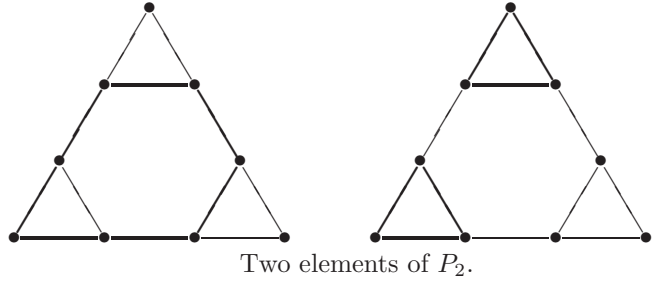
\begin{center}
\begin{picture}(400,120)
\letvertex A=(80,10)\letvertex B=(100,44)

\letvertex C=(120,78)\letvertex D=(140,112)

\letvertex E=(160,78)\letvertex F=(180,44)

\letvertex G=(200,10)\letvertex H=(160,10)\letvertex I=(120,10)

\drawvertex(A){$\bullet$}\drawvertex(B){$\bullet$}

\drawvertex(C){$\bullet$}\drawvertex(D){$\bullet$}

\drawvertex(E){$\bullet$}\drawvertex(F){$\bullet$}

\drawvertex(G){$\bullet$}\drawvertex(H){$\bullet$}

\drawvertex(I){$\bullet$}

\thicklines \drawundirectededge(A,B){}\drawundirectededge(B,C){}

\drawundirectededge(E,C){}\drawundirectededge(H,F){}

\drawundirectededge(E,F){}\drawundirectededge(H,I){}\drawundirectededge(I,A){}

\thinlines \drawundirectededge(C,D){}

\drawundirectededge(D,E){}\drawundirectededge(G,H){}

\drawundirectededge(F,G){}\drawundirectededge(B,I){}

\letvertex a=(240,10)\letvertex b=(260,44)

\letvertex c=(280,78)\letvertex d=(300,112)

\letvertex e=(320,78)\letvertex f=(340,44)

\letvertex g=(360,10)\letvertex h=(320,10)\letvertex i=(280,10)

\drawvertex(a){$\bullet$}\drawvertex(b){$\bullet$}

\drawvertex(c){$\bullet$}\drawvertex(d){$\bullet$}

\drawvertex(e){$\bullet$}\drawvertex(f){$\bullet$}

\drawvertex(g){$\bullet$}\drawvertex(h){$\bullet$}

\drawvertex(i){$\bullet$}

\thicklines

\drawundirectededge(c,d){}\drawundirectededge(d,e){}\drawundirectededge(e,c){}

\drawundirectededge(i,a){}

\drawundirectededge(b,i){}\drawundirectededge(a,b){}

\thinlines

\drawundirectededge(g,h){}\drawundirectededge(e,f){}\drawundirectededge(b,c){}\drawundirectededge(h,f){}

\drawundirectededge(f,g){} \drawundirectededge(h,i){}

\put(180,-5){Two elements of $P_2$.}
\end{picture}
\end{center}
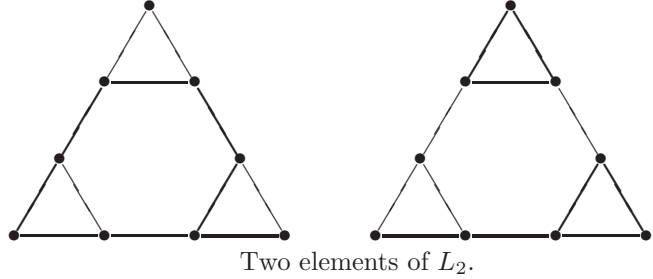
\begin{center}
\begin{picture}(400,120)
\letvertex A=(80,10)\letvertex B=(100,44)

\letvertex C=(120,78)\letvertex D=(140,112)

\letvertex E=(160,78)\letvertex F=(180,44)

\letvertex G=(200,10)\letvertex H=(160,10)\letvertex I=(120,10)

\drawvertex(A){$\bullet$}\drawvertex(B){$\bullet$}

\drawvertex(C){$\bullet$}\drawvertex(D){$\bullet$}

\drawvertex(E){$\bullet$}\drawvertex(F){$\bullet$}

\drawvertex(G){$\bullet$}\drawvertex(H){$\bullet$}

\drawvertex(I){$\bullet$}

\thicklines  \drawundirectededge(E,C){}\drawundirectededge(H,F){}

\drawundirectededge(A,B){}\drawundirectededge(B,C){}\drawundirectededge(E,F){}\drawundirectededge(G,H){}

\put(180,-5){Two elements of $L_2$.}

\thinlines \drawundirectededge(C,D){} \drawundirectededge(D,E){}

\drawundirectededge(F,G){}\drawundirectededge(B,I){}

\drawundirectededge(H,I){}\drawundirectededge(I,A){}

\letvertex a=(240,10)\letvertex b=(260,44)

\letvertex c=(280,78)\letvertex d=(300,112)

\letvertex e=(320,78)\letvertex f=(340,44)

\letvertex g=(360,10)\letvertex h=(320,10)\letvertex i=(280,10)

\drawvertex(a){$\bullet$}\drawvertex(b){$\bullet$}

\drawvertex(c){$\bullet$}\drawvertex(d){$\bullet$}

\drawvertex(e){$\bullet$}\drawvertex(f){$\bullet$}

\drawvertex(g){$\bullet$}\drawvertex(h){$\bullet$}

\drawvertex(i){$\bullet$}

\thicklines  \drawundirectededge(i,a){}\drawundirectededge(h,f){}

\drawundirectededge(f,g){}

\drawundirectededge(h,i){}\drawundirectededge(d,e){}\drawundirectededge(c,d){}\drawundirectededge(e,c){}

\thinlines \drawundirectededge(b,i){}

\drawundirectededge(g,h){}\drawundirectededge(a,b){}\drawundirectededge(e,f){}\drawundirectededge(b,c){}
\end{picture}
\end{center}
\vspace{0,5 cm} Each closed polygon in $\Sigma_n$ can be obtained
in the following way: either it is a union of closed polygons
living in the three copies $\Sigma_{n-1}$ or it contains the three
special edges joining the three subgraphs isomorphic to
$\Sigma_{n-1}$. The subgraphs of the first type can be identified
with the elements of the set $P_{n-1}^3$, whereas the other ones
are obtained by joining three elements in $L_{n-1}$, each one
belonging to one of the three copies of $\Sigma_{n-1}$, so that
they can be identified with elements of the set $L_{n-1}^3$. This
gives
\begin{eqnarray}\label{ricorrenza per ising}
P_n=P_{n-1}^3 \coprod L_{n-1}^3.
\end{eqnarray}
On the other hand, each element in $L_n$ can be described in the
following way: if it contains a path that does not reach the
up-most triangle isomorphic to $\Sigma_{n-1}$, it can be regarded
as an element in $L_{n-1}^2\times P_{n-1}$; if it contains a path
which goes through all  three copies of $\Sigma_{n-1}$, then it is
in $L_{n-1}^3$. This gives
\begin{eqnarray}\label{ricorrenza per ising2}
L_{n}=\left(L_{n-1}^2 \times P_{n-1}\right) \coprod L_{n-1}^3,
\end{eqnarray}
from which we deduce
\begin{prop}
For each $n\geq 1$ the number $|P_n|$ of closed polygons in the
Schreier graph $\Sigma_n$ of $H^{(3)}$ is $2^{\frac{3^n-1}{2}}$.
\end{prop}
We are now ready to compute the generating series for closed
polygons and the partition function of the Ising model on Schreier
graphs of $H^{(3)}$. Denote by $\Gamma_n^{cl}(z)$ the generating
function of the set of subgraphs in $P_n$ and by $\Upsilon_n(z)$
the generating function of the set of subgraphs in $L_n$. The
equation \eqref{ricorrenza per ising} gives
\begin{eqnarray}\label{ricorrenza per funzione generatrice hanoi}
\Gamma_n^{cl}(z)=\left(\Gamma_{n-1}^{cl}(z)\right)^3+z^3\Upsilon_{n-1}^3(z).
\end{eqnarray}
The factor $z^3$ in \eqref{ricorrenza per funzione generatrice
hanoi} is explained by the fact that each term in
$\Upsilon_n^3(z)$ corresponds to a set of edges that becomes a
closed polygon after adding the three special edges connecting the
three copies of $\Sigma_{n-1}$.  We have consequently that the
second summand is the generating function for the closed polygons
containing the three special edges. Analogously, from
\eqref{ricorrenza per ising2} we have
\begin{eqnarray}\label{secondautile}
\Upsilon_n(z)=z\Upsilon_{n-1}^2(z)\Gamma_{n-1}^{cl}(z)+z^2\Upsilon_{n-1}^3(z).
\end{eqnarray}
\begin{teo}\label{fiisinghanoi}
For each $n\geq 1$, the partition function of the Ising model on
the Schreier graph $\Sigma_n$ of the group $H^{(3)}$ is
$$
Z_n=2^{3^n}\cdot\cosh(\beta
J)^\frac{3^{n+1}-3}{2}\cdot\Gamma_n^{cl}(\tanh(\beta J)),
$$
with
$$
\Gamma_n^{cl}(z)=z^{3^n}\prod_{k=1}^n\psi_k^{3^{n-k}}(z)\cdot(\psi_{n+1}(z)-1),
$$
where $\psi_1(z) = \frac{z+1}{z}$ and $\psi_k(z) =
\psi_{k-1}^2(z)-3\psi_{k-1}(z)+4$, for each $k\geq 2$.
\end{teo}
\begin{proof}
Recall that $Z_n = 2^{|Vert(\Sigma_n)|}\cosh(\beta
J)^{|Edges(\Sigma_n)|}\cdot\Gamma_n^{cl}(\tanh(\beta J))$, where
$\Gamma_n^{cl}(z)$ is the generating function of closed polygons
in $\Sigma_n$. In our case we have $|Edges(\Sigma_n)|\\ =
\frac{3^{n+1}-3}{2}$ and $|Vert(\Sigma_n)|= 3^n$.

We know that the generating functions $\Gamma_n^{cl}(z)$ and
$\Upsilon_n(z)$ satisfy equations \eqref{ricorrenza per funzione
generatrice hanoi} and \eqref{secondautile}, and the initial
conditions can be easily computed as:
$$
\Gamma_1^{cl}(z)=1+z^3 \qquad \Upsilon_1(z)=z^2+z.
$$
We now show by induction on $n$ that the solutions of the system
of equations \eqref{ricorrenza per funzione generatrice hanoi} and
\eqref{secondautile} are
$$
\begin{cases}
\Gamma_n^{cl}(z)=z^{3^n}\prod_{k=1}^n\psi_k^{3^{n-k}}(z)\cdot(\psi_{n+1}(z)-1)\\
\Upsilon_n(z)=z^{3^n-1}\prod_{k=1}^n\psi_k^{3^{n-k}}(z).
\end{cases}
$$
For $n=1$, we get
$\Gamma_1^{cl}(z)=z^3\psi_1(z)(\psi_2(z)-1)=z^3+1$ and
$\Upsilon_1(z)= z^2\psi_1(z)=z^2+z$ and so the claim is true. Now
suppose that the assertion is true for $n$ and let us show that it
is true for $n+1$. One gets:
\begin{eqnarray*}
\Gamma_{n+1}^{cl}(z) &=&
\left(z^{3^n}\prod_{k=1}^n\psi_k^{3^{n-k}}(z)\cdot(\psi_{n+1}(z)-1)\right)^3+z^3\left(z^{3^n-1}\prod_{k=1}^n\psi_k^{3^{n-k}}(z)\right)^3\\
&=&z^{3^{n+1}}\prod_{k=1}^n\psi_k^{3^{n-k+1}}(z)\left(\psi_{n+1}^3(z)-3\psi_{n+1}^2(z)+3\psi_{n+1}(z)\right)\\
&=&z^{3^{n+1}}\prod_{k=1}^{n+1}\psi_k^{3^{n-k+1}}(z)(\psi_{n+2}(z)-1)
\end{eqnarray*}
and
\begin{eqnarray*}
\Upsilon_{n+1}(z) &=&
z\left(z^{3^n-1}\prod_{k=1}^n\psi_k^{3^{n-k}}(z)\right)^2\left(z^{3^n}\prod_{k=1}^n\psi_k^{3^{n-k}}(z)\cdot(\psi_{n+1}(z)-1)\right)\\
&+&z^2\left(z^{3^n-1}\prod_{k=1}^n\psi_k^{3^{n-k}}(z)\right)^3\\
&=& z^{3^{n+1}-1}\prod_{k=1}^n\psi_k^{3^{n-k+1}}(z)\cdot \psi_{n+1}(z)\\
&=& z^{3^{n+1}-1}\prod_{k=1}^{n+1}\psi_k^{3^{n-k+1}}(z).
\end{eqnarray*}
\end{proof}
\begin{teo}\label{EXISTENCEhanoi}
The thermodynamic limit $ \lim_{n\to
\infty}\frac{\log(Z_n)}{|Vert(\Sigma_n)|}$ exists.
\end{teo}
\begin{proof}
Since $|Edges(\Sigma_n)| = \frac{3^{n+1}-3}{2}$ and
$|Vert(\Sigma_n)|= 3^n$, the limit reduces to
$$
\log(2) + \frac{3}{2}\log(\cosh(\beta J)) +
\lim_{n\to\infty}\frac{\log(\Gamma_n^{cl}(z))}{3^n},
$$
where $\tanh (\beta J)$ takes values between 0 and 1. It is
straightforward to show, by induction, that
$\psi_k(z)=\frac{\varphi_k(z)}{z^{2^{k-1}}}$, for every $k\geq 1$,
where $\varphi_k(z)$ is a polynomial of degree $2^{k-1}$ in $z$.
Hence, the limit
$\lim_{n\to\infty}\frac{\log(\Gamma_n^{cl}(z))}{3^n}$ becomes
$$
\lim_{n\to \infty}\frac{\log
\left(\prod_{k=1}^n\varphi_k^{3^{n-k}}(z)\cdot
\left(\varphi_{n+1}(z)-z^{2^n}\right)\right)}{3^n} =
$$
$$
\lim_{n\to \infty}\sum_{k=1}^n\frac{\log(\varphi_k(z))}{3^k} +
\lim_{n\to \infty}\frac{\log(\varphi_{n+1}(z)-z^{2^n})}{3^n}.
$$
Let us show that the series
$\sum_{k=1}^{\infty}\frac{\log(\varphi_k(z))}{3^k}$ converges
absolutely, and that\\ $\lim_{n\to
\infty}\frac{\log(\varphi_{n+1}(z)-z^{2^n})}{3^n}=0$. It is not
difficult to show by induction that
$$
2z^{2^{k-1}}\leq \varphi_k(z)\leq 2^{2^k-1}
$$
for each $k\geq 2$ and $z\in [0,1]$, so that
$$
|\log(\varphi_k(z))|\leq \max\{2^{k-1}|\log(z)|,
(2^k-1)|\log(2)|\}.
$$
Note that $1\leq \varphi_1(z)\leq 2$ for each $0\leq z \leq 1$.
Moreover, one can directly verify that $\varphi_2(z)$ has a
minimum at $c_2=1/4$ and $\varphi_2'(z)<0$ for each $z\in
(0,c_2)$. Let us call $c_k$ the point where $\varphi_k(z)$ has the
first minimum. One can prove by induction that
$\varphi'_{k+1}(z)<0$ for each $z\in (0,c_k]$ and so $c_k<c_{k+1}$
for every $k\geq 2$. In particular, $\varphi_k(z)$ satisfies
$$
-2^k\log (2) \leq \log(\varphi_k(z)) \leq (2^k-1)\log (2),
$$
that gives $|\log(\varphi_k(z))|\leq 2^k\log(2)$ for each $k\geq
2$. So we can conclude that
$$
\sum_{k=1}^{\infty} \frac{|\log(\varphi_k(z))|}{3^k} \leq
\frac{\log(2)}{3} + \sum_{k=2}^{\infty}\frac{2^k\log(2)}{3^k}<
\infty.
$$
Moreover $\lim_{n\to
\infty}\frac{|\log(\varphi_{n+1}(z)-z^{2^n})|}{3^n}\leq \lim_{n\to
\infty}\frac{2^n\log(2)}{3^n}=0$.
\end{proof}

\subsection{The Sierpi\'{n}ski gasket}\label{chapterSierpinski}

In this section we use the high temperature expansion and counting
of closed polygons in order to compute the partition function for
the Ising model on a sequence of graphs $\{\Omega_n\}_{n\geq 1}$
converging to the Sierpi\'{n}ski gasket. The graphs $\Omega_n$ are
close relatives the Schreier graphs $\Sigma_n$ of the group
$H^{(3)}$ considered above. More precisely, one can obtain
$\Omega_n$ from $\Sigma_n$ by contracting the edges between copies
of $\Sigma_{n-1}$ in $\Sigma_n$. The graphs $\Omega_n$ are also
self-similar in the sense of \cite{wagner2}, as can be seen in the
picture.
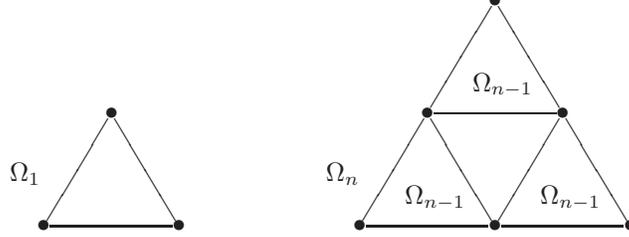
\begin{center}
\begin{picture}(400,105)
\put(55,30){$\Omega_1$}\put(195,30){$\Omega_n$}

\letvertex A=(100,60)\letvertex B=(70,10)\letvertex C=(130,10)

\letvertex D=(270,110)\letvertex E=(240,60)\letvertex F=(210,10)\letvertex G=(270,10)

\letvertex H=(330,10)\letvertex I=(300,60)

\put(260,70){$\Omega_{n-1}$}\put(230,20){$\Omega_{n-1}$}\put(290,20){$\Omega_{n-1}$}

\drawvertex(A){$\bullet$}\drawvertex(B){$\bullet$}

\drawvertex(C){$\bullet$}\drawvertex(D){$\bullet$}

\drawvertex(E){$\bullet$}\drawvertex(F){$\bullet$}

\drawvertex(G){$\bullet$}\drawvertex(H){$\bullet$}

\drawvertex(I){$\bullet$}

\drawundirectededge(B,A){} \drawundirectededge(C,B){}

\drawundirectededge(A,C){} \drawundirectededge(E,D){}

\drawundirectededge(F,E){} \drawundirectededge(G,F){}

\drawundirectededge(H,G){} \drawundirectededge(I,H){}

\drawundirectededge(D,I){} \drawundirectededge(I,E){}

\drawundirectededge(E,G){} \drawundirectededge(G,I){}
\end{picture}
\end{center}
Similarly to the case of $H^{(3)}$ above, define sets $P_n$ and
$L_n$. The same recursive rules hold, and  the total number of
closed polygons is again $2^{\frac{3^n-1}{2}}$, since the initial conditions are the same.\\
\indent Let $\Gamma_n^{cl}(z)$ denote the generating function of
the subgraphs in $P_n$ and let $\Upsilon_n(z)$ denote the
generating function of the subgraphs in $L_n$. From relations
\eqref{ricorrenza per ising} and \eqref{ricorrenza per ising2} we
deduce the following formulae:
\begin{eqnarray}\label{ricorrenza per funzione generatrice}
\Gamma_n^{cl}(z)=\left(\Gamma_{n-1}^{cl}(z)\right)^3+\Upsilon_{n-1}^3(z).
\end{eqnarray}
and
\begin{eqnarray}\label{ricoriconuova}
\Upsilon_n(z)=\Upsilon_{n-1}^2(z)\Gamma_{n-1}^{cl}(z)+\Upsilon_{n-1}^3(z).
\end{eqnarray}
Note that in \eqref{ricorrenza per funzione generatrice} and
\eqref{ricoriconuova} there are no factors $z,z^2,z^3$ occurring
in \eqref{ricorrenza per funzione generatrice hanoi} and
\eqref{secondautile}, because the special edges connecting
elementary triangles have been contracted in $\Omega_n$.
\begin{teo}\label{teosierpinskigenfunct}
For each $n\geq 1$, the partition function of the Ising model on
the $n$-th Sierpi\'nski graph $\Omega_n$ is
$$
Z_n= 2^{\frac{3^n+3}{2}}\cdot\cosh(\beta
J)^{3^n}\cdot\Gamma_n^{cl}(\tanh(\beta J)),
$$
with
$$
\Gamma_n^{cl}(z)=z^{\frac{3^n}{2}}\prod_{k=1}^n\psi_k^{3^{n-k}}(z)\cdot(\psi_{n+1}(z)-1),
$$
where $\psi_1(z) = \frac{z+1}{z^{1/2}}$,
$\psi_2(z)=\frac{z^2+1}{z}$ and
$\psi_k(z)=\psi_{k-1}^2(z)-3\psi_{k-1}(z)+4$, for each $k\geq 3$.
\end{teo}
\begin{proof}
Again we shall use the expression $Z_n =
2^{|Vert(\Omega_n)|}\cosh(\beta
J)^{|Edges(\Omega_n)|}\cdot\Gamma_n^{cl}(\tanh(\beta J))$, where
$\Gamma_n^{cl}(z)$ is the generating function of closed polygons
in $\Omega_n$. In our case we have $|Edges(\Omega_n)| = 3^n$ and
$|Vert(\Omega_n)|= \frac{3^n+3}{2}$.

We know that the generating functions $\Gamma_n^{cl}(z)$ and
$\Upsilon_n(z)$ satisfy the equations \eqref{ricorrenza per
funzione generatrice} and \eqref{ricoriconuova}, with the initial
conditions
$$
\Gamma_1^{cl}(z)=1+z^3 \qquad \Upsilon_1(z)=z^2+z.
$$
Let us show by induction on $n$ that the solutions of the system
of equations \eqref{ricorrenza per funzione generatrice} and
\eqref{ricoriconuova} are
$$
\begin{cases}
\Gamma_n^{cl}(z)=z^{\frac{3^n}{2}}\prod_{k=1}^n\psi_k^{3^{n-k}}(z)\cdot
(\psi_{n+1}(z)-1)\\
\Upsilon_n(z)=z^{\frac{3^n}{2}}\prod_{k=1}^n\psi_k^{3^{n-k}}(z).
\end{cases}
$$
For $n=1$, we get
$\Gamma_1^{cl}(z)=z^{\frac{3}{2}}\psi_1(z)(\psi_2(z)-1)=z^3+1$ and
$\Upsilon_1(z)= z^{\frac{3}{2}}\psi_1(z)=z^2+z$ and so the claim
is true. Now suppose that the assertion is true for $n$ and let us
show that it is true for $n+1$. One gets:
\begin{eqnarray*}
\Gamma_{n+1}^{cl}(z) &=&
\left(z^{\frac{3^n}{2}}\prod_{k=1}^n\psi_k^{3^{n-k}}(z)\cdot(\psi_{n+1}(z)-1)\right)^3+\left(z^{\frac{3^n}{2}}\prod_{k=1}^n\psi_k^{3^{n-k}}(z)\right)^3\\
&=&z^{\frac{3^{n+1}}{2}}\prod_{k=1}^n\psi_k^{3^{n-k+1}}(z)\left(\psi_{n+1}^3(z)-3\psi_{n+1}^2(z)+3\psi_{n+1}(z)\right)\\
&=&z^{\frac{3^{n+1}}{2}}\prod_{k=1}^{n+1}\psi_k^{3^{n-k+1}}(z)\cdot\left(\psi_{n+2}(z)-1\right)
\end{eqnarray*}
and
\begin{eqnarray*}
\Upsilon_{n+1}(z)
&=&\left(z^{\frac{3^n}{2}}\prod_{k=1}^n\psi_k^{3^{n-k}}(z)\right)^2\left(z^{\frac{3^n}{2}}\prod_{k=1}^n\psi_k^{3^{n-k}}(z)\cdot
(\psi_{n+1}(z)-1)\right)\\
&+&\left(z^{\frac{3^n}{2}}\prod_{k=1}^n\psi_k^{3^{n-k}}(z)\right)^3\\
&=& z^{\frac{3^{n+1}}{2}}\prod_{k=1}^n\psi_k^{3^{n-k+1}}(z)\cdot \psi_{n+1}(z)\\
&=&z^{\frac{3^{n+1}}{2}}\prod_{k=1}^{n+1}\psi_k^{3^{n-k+1}}(z).
\end{eqnarray*}
\end{proof}
\begin{os}\rm
The existence of the thermodynamic limit can be shown in exactly
the same way as for the Schreier graphs of the Hanoi Towers group.
\end{os}

\subsection{Renormalization approach}
Expressions for the partition function of the Ising model on the
Sierpi\'nski gasket are well known to physicists. A
renormalization equation for it can be found for example in
\cite{Mand} (see also references therein), and a more detailed
analysis is given in \cite{burioni}. Using renormalization,
Burioni et al \cite{burioni} give the following recursion for the
partition function of the Ising model on the graphs $\Omega_n$,
$n\geq 1$:
\begin{eqnarray}\label{burioni2}
Z_{n+1}(y)=Z_n(f(y))[c(y)]^{3^{n-1}},
\end{eqnarray}
where $y=\exp(\beta J)$; $f(y)$ is a substitution defined by $$y
\rightarrow f(y)=\left(\frac{y^8-y^4+4}{y^4+3}\right)^{1/4};$$ and
$$
c(y)=\frac{y^4+1}{y^3} [(y^4+3)^3(y^8-y^4+4)]^{1/4};
$$
with
$$
Z_1(y)=2y^3+6y^{-1}\ .
$$
A similar computation can be performed in the case of the Schreier
graph $\Sigma_n$ of the group $H^{(3)}$, where the self-similarity
of the graph allows to compare the partition function $Z_1$ of the
first level with the partition function of level 2, where the sum
is taken only over the internal spins
$\sigma_2,\sigma_3,\sigma_5,\sigma_6,\sigma_8,\sigma_9$ (see
figure below). The resulting recurrence is the same as
\eqref{burioni2}, but with
$$
y \rightarrow
f(y)=\left(\frac{y^8-2y^6+2y^4+2y^2+1}{2(y^4+1)}\right)^{1/4}
$$
and
$$
c(y)=\frac{(y^4-y^2+2)(y^2+1)^3}{y^6}(8(y^4+1)^3(y^8-2y^6+2y^4+2y^2+1))^{1/4}\ .
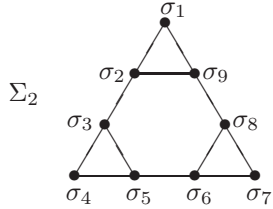
$$
\unitlength=0,2mm
\begin{center}
\begin{picture}(400,115)
\letvertex A=(120,10)\letvertex B=(140,44)

\letvertex C=(160,78)\letvertex D=(180,112)

\letvertex E=(200,78)\letvertex F=(220,44)

\letvertex G=(240,10)\letvertex H=(200,10)\letvertex I=(160,10)

\put(76,60){$\Sigma_2$}

\put(177,118){$\sigma_1$}

\put(136,75){$\sigma_2$}\put(115,43){$\sigma_3$}

\put(114,-6){$\sigma_4$}\put(154,-6){$\sigma_5$}

\put(194,-6){$\sigma_6$}\put(234,-6){$\sigma_7$}

\put(225,43){$\sigma_8$}\put(205,75){$\sigma_9$}

\drawvertex(A){$\bullet$}\drawvertex(B){$\bullet$}

\drawvertex(C){$\bullet$}\drawvertex(D){$\bullet$}

\drawvertex(E){$\bullet$}\drawvertex(F){$\bullet$}

\drawvertex(G){$\bullet$}\drawvertex(H){$\bullet$}

\drawvertex(I){$\bullet$}

\drawundirectededge(A,B){}\drawundirectededge(B,C){}\drawundirectededge(C,D){}

\drawundirectededge(D,E){}\drawundirectededge(E,C){}\drawundirectededge(E,F){}\drawundirectededge(F,G){}

\drawundirectededge(B,I){}\drawundirectededge(H,F){}\drawundirectededge(H,I){}

\drawundirectededge(I,A){}\drawundirectededge(G,H){}
\end{picture}
\end{center}
\begin{os}\rm The above recursions for the partition functions for $\Omega_n$'s
and $\Sigma_n$'s can be deduced from our Theorems
\ref{fiisinghanoi} and \ref{teosierpinskigenfunct} by rewriting
the formulae in the variable $y=\exp(\beta J)$ and substituting
$z=\tanh(\beta J)=(y^2-1)/(y^2+1)$.
\end{os}


\section{Statistics on weighted closed polygons}\label{chapterstatitics}

This Section is devoted to the study of weighted generating
functions of closed polygons, i.e., we allow the edges of the
graph to have different weights $\tanh(\beta J_{i,j})$, as in RHS
of \eqref{polygons}. We also take into account the fact that the
graphs we consider are Schreier graphs of some self-similar group
$G$ with respect to a certain generating set $S$, and their edges
are therefore labeled by these generators. It is thus natural to
allow the situations where the energy between two neighbouring
spins takes a finite number of possible values encoded by the
generators $S$. Logarithmic derivatives of the weighted generating
function with respect to $s\in S$  give us the mean density of
$s$-edges in a random configuration. We can further find the
variance and show that the limiting distribution is normal.

\subsection{The Schreier graphs of the Grigorchuk's group}\label{statisticsweightedgrig}

Recall from \ref{gri} that the simple edges in $\Sigma_n$ are
always labeled by $a$. Moreover, the $2$-cycles can be labeled by
the couples of labels $(b,c)$, $(b,d)$ and $(c,d)$.  We want to
compute the weighted generating function of closed polygons, with
respect to the weights given by the labels $a,b,c,d$. Let us set,
for each $n\geq 1$:
$$
X_n=|\{2\textrm{-cycles with labels }b,c\}| \ \ \ \
Y_n=|\{2\textrm{-cycles with labels }b,d\}|
$$
$$
W_n=|\{2\textrm{-cycles with labels }c,d\}|
$$
One can easily check by using self-similar formulae for the
generators, that the following equations hold:
$$
\begin{cases}
X_n = W_{n-1}+2^{n-2}\\
Y_n = X_{n-1}\\
W_n = Y_{n-1}.
\end{cases}
$$
In particular, one gets
$$
\begin{cases}
X_n = X_{n-3}+2^{n-2}\\
Y_n = X_{n-1}\\
W_n = X_{n-2},
\end{cases}
$$
with initial conditions $X_1=0$, $X_2=1$ and $X_3=2$. One gets the
following values:
$$
X_n=
\begin{cases}
\frac{2^{n+1}-2}{7}&\text{if } n\equiv 0(3)\\
\frac{2^{n+1}-4}{7}&\text{if } n\equiv 1(3)\\
\frac{2^{n+1}-1}{7}&\text{if } n\equiv 2(3)
\end{cases} \qquad
Y_n=
\begin{cases}
\frac{2^{n}-1}{7}&\text{if } n\equiv 0(3)\\
\frac{2^{n}-2}{7}&\text{if } n\equiv 1(3)\\
\frac{2^{n}-4}{7}&\text{if } n\equiv 2(3)
\end{cases}
$$
$$
W_n=
\begin{cases}
\frac{2^{n-1}-4}{7}&\text{if } n\equiv 0(3)\\
\frac{2^{n-1}-1}{7}&\text{if } n\equiv 1(3)\\
\frac{2^{n-1}-2}{7}&\text{if } n\equiv 2(3)
\end{cases}\ ,
$$
and, consequently,
\begin{teo}
For each $n\geq 1$, the weighted generating function of closed
polygons in $\Sigma_n$ is
$$
\begin{cases}
\Gamma_n^{cl}(a,b,c,d) = (1+bc)^{\frac{2^{n+1}-2}{7}}(1+bd)^{\frac{2^n-1}{7}}(1+cd)^{\frac{2^{n-1}-4}{7}} & \text{if} \ n\equiv 0(3)\\
\Gamma_n^{cl}(a,b,c,d) = (1+bc)^{\frac{2^{n+1}-4}{7}}(1+bd)^{\frac{2^n-2}{7}}(1+cd)^{\frac{2^{n-1}-1}{7}} & \text{if} \ n\equiv 1(3)\\
\Gamma_n^{cl}(a,b,c,d) =
(1+bc)^{\frac{2^{n+1}-1}{7}}(1+bd)^{\frac{2^n-4}{7}}(1+cd)^{\frac{2^{n-1}-2}{7}}
& \text{if} \ n\equiv 2(3)\ .
\end{cases}
$$
\end{teo}
\begin{prop} Let $w_n$ be the number of edges
labeled $w$ in a random closed polygon in $\Sigma_n$, where
$w=a,b,c,d$. Denote by $\mu_{n,w}$ and $\sigma^2_{n,w}$ the mean
and the variance of $w_n$. Then,
\begin{itemize}
\item for each $n\geq 1$,  $a_n=0$;
\item The means and the
variances of the random variables $b_n,c_n,d_n$ are given in the
following table:
\begin{center}
\begin{tabular}{|c|c|c|c|}
\hline
   & $n\equiv 0(3)$   &  $n\equiv 1(3)$  &   $n\equiv 2(3)$ \\
\hline $\mu_{n,b}$ & $\frac{3}{14}(2^n-1)$ &
$\frac{3}{7}(2^{n-1}-1)$ & $\frac{3\cdot 2^n-5}{14}$\\
\hline $\sigma_{n,b}^2$ & $\frac{3}{28}(2^n-1)$ &
$\frac{3}{14}(2^{n-1}-1)$ & $\frac{3\cdot 2^n-5}{28}$\\
\hline $\mu_{n,c}$ & $\frac{5\cdot 2^{n-2}-3}{7}$ &
$\frac{5}{14}(2^{n-1}-1)$ & $\frac{5\cdot 2^{n-1}-3}{14}$\\
\hline $\sigma_{n,c}^2$ & $\frac{5\cdot 2^{n-2}-3}{14}$ &
$\frac{5}{28}(2^{n-1}-1)$ & $\frac{5\cdot 2^{n-1}-3}{28}$\\
\hline $\mu_{n,d}$ & $\frac{3\cdot2^{n-1}-5}{14}$ &
$\frac{3}{14}(2^{n-1}-1)$ & $\frac{3}{7}(2^{n-2}-1)$\\
\hline $\sigma_{n,d}^2$ & $\frac{3\cdot2^{n-1}-5}{28}$ &
$\frac{3}{28}(2^{n-1}-1)$ & $\frac{3}{14}(2^{n-2}-1)$\\
\hline
\end{tabular}
\end{center}
\item the random variables $b_n,c_n,d_n$ are asymptotically normal, as $n\rightarrow \infty$.
\end{itemize}
\end{prop}
\begin{proof}
It is clear that an edge labeled by $a$ never belongs to a closed
polygon of $\Sigma_n$, so that $a_n=0$. Let us choose, for
instance, $n\equiv 0(3)$. Putting
$$
\Gamma_n^{cl}(b):=\Gamma_n^{cl}(1,b,1,1) =
2^{\frac{2^{n-1}-4}{7}}(1+b)^{\frac{3\cdot 2^n-3}{7}},
$$
we can obtain the mean $\mu_{n,b}$ and the variance
$\sigma^2_{n,b}$ for $b_n$ by studying the derivatives of the
function $\log(\Gamma_n^{cl}(b))$. One gets
$$
\mu_{n,b}=\frac{3}{14}(2^n-1) \qquad
\sigma^2_{n,b}=\frac{3}{28}(2^n-1).
$$
Similar computations can be done for $c$ and $d$.

In all cases above we can find explicitly the moment generating
function of the corresponding normalized random variables: a
direct computation of its limit for $n\rightarrow \infty$ shows
that the asymptotic distribution is normal.
\end{proof}
\subsection{The Schreier graphs of the Basilica group}
We also compute the weighted generating function of closed
polygons for the Basilica group, with respect to the weights given
by the labels $a$ and $b$ on the edges of its Schreier graph
$\Sigma_n$. We use here the computations from Proposition
\ref{bascycles}.
\begin{teo}
The weighted generating function of closed polygons in the
Schreier graph $\Sigma_n$ of the Basilica group is
\begin{eqnarray*}
\Gamma_n^{cl}(a,b) &=&
\prod_{k=1}^{\frac{n-1}{2}-1}\left(1+a^{2^k}\right)^{2^{n-2k-1}}\prod_{k=1}^{\frac{n-1}{2}-1}\left(1+b^{2^k}\right)^{2^{n-2k}}\\
&\cdot&\left(1+a^{2^{\frac{n-1}{2}}}\right)^2\left(1+b^{2^{\frac{n-1}{2}}}\right)^2\left(1+b^{2^{\frac{n+1}{2}}}\right)
\end{eqnarray*}
for $n\geq 5$ odd and
$$
\Gamma_n^{cl}(a,b) =
\prod_{k=1}^{\frac{n}{2}-1}\left(1+a^{2^k}\right)^{2^{n-2k-1}}\prod_{k=1}^{\frac{n}{2}-1}\left(1+b^{2^k}\right)^{2^{n-2k}}
\left(1+a^{2^{\frac{n}{2}}}\right)\left(1+b^{2^{\frac{n}{2}}}\right)^2
$$
for $n\geq 4$ even.
\end{teo}
\begin{prop}
The means and the variances of the densities $a_n$ and $b_n$ are
given in the following table:
\begin{center}
\begin{tabular}{|c|c|c|}
\hline
     & $n\geq 5$ odd   &  $n\geq 4$ even\\
\hline $\mu_{n,a}$ & $2^{n-2}$ & $2^{n-2}$\\
\hline $\sigma_{n,a}^2$ & $(n+1)2^{n-4}$ & $(n+2)2^{n-4}$\\
\hline $\mu_{n,b}$ & $2^{n-1}$ & $2^{n-1}$\\
\hline $\sigma_{n,b}^2$ & $(n+3)2^{n-3}$ & $(n-2)2^{n-3}$\\
\hline
\end{tabular}
\end{center}
\end{prop}

\subsection{The Schreier graphs of $H^{(3)}$}\label{hanoistatisticsfeb}
Let us denote by $\Upsilon_n^{lr}(a,b,c)$ the weighted generating
function of the subgraphs that belong to the set $P_n$, defined in
Subsection \ref{grafohanoigrafo} (the exponent $lr$ stands for
left-right, as self-avoiding paths in $P_n$ join the left-most and
the right-most vertices of $\Sigma_n$.) Analogously, we define
$\Upsilon_n^{lu}(a,b,c)$ and $\Upsilon_n^{ru}(a,b,c)$, where the
exponents $lu$ and $ru$ stand for left-up and right-up,
respectively. By using the self-similar expressions for the
generators given in Subsection \ref{grafohanoigrafo}, we find that
these functions satisfy the following system of equations (we omit
the arguments $a,b,c$):
\begin{eqnarray}\label{easysystem}
\begin{cases}
\Gamma_{n+1}^{cl} = \left(\Gamma_n^{cl}\right)^3+abc\Upsilon_n^{lr}\Upsilon_n^{lu}\Upsilon_n^{ru}\\
\Upsilon_{n+1}^{lu} =a\Upsilon_n^{lr}\Upsilon^{ru}_n\Gamma_n^{cl}
+bc\left(\Upsilon_n^{lu}\right)^3\\
\Upsilon_{n+1}^{ru} =c\Upsilon_n^{lu}\Upsilon^{lr}_n\Gamma_n^{cl}
+ab\left(\Upsilon_n^{ru}\right)^3\\
\Upsilon_{n+1}^{lr} =b\Upsilon_n^{lu}\Upsilon^{ru}_n\Gamma_n^{cl}
+ac\left(\Upsilon_n^{lr}\right)^3\\
\end{cases}
\end{eqnarray}
with the initial conditions $\Gamma_1^{cl}(a,b,c)=1+abc$,
$\Upsilon_1^{lr}(a,b,c) = ac+b$, $\Upsilon_1^{lu}(a,b,c)=a+bc$,
$\Upsilon_1^{ru}(a,b,c)=c+ab$.
\begin{prop}
The mean and the variance for $w_n$, with $w=a,b,c$, are:
$$
\mu_{n,w}=\frac{3^n-1}{4} \ \ \ \ \ \ \
\sigma_{n,w}^2=\frac{3^n-1}{8}.
$$
The random variables $w_n$ with $w=a,b,c$ are asymptotically
normal.
\end{prop}
\begin{proof}
If we put $a=b=1$, the system \eqref{easysystem} reduces to
\begin{eqnarray}\label{easysystemeasy}
\begin{cases}
\Gamma_{n+1}^{cl} = \left(\Gamma_n^{cl}\right)^3+c\Upsilon_n^{lr}\Upsilon_n^{lu}\Upsilon_n^{ru}\\
\Upsilon_{n+1}^{lu} =\Upsilon_n^{lr}\Upsilon^{ru}_n\Gamma_n^{cl}
+c\left(\Upsilon_n^{lu}\right)^3\\
\Upsilon_{n+1}^{ru} =c\Upsilon_n^{lu}\Upsilon^{lr}_n\Gamma_n^{cl}
+\left(\Upsilon_n^{ru}\right)^3\\
\Upsilon_{n+1}^{lr} =\Upsilon_n^{lu}\Upsilon^{ru}_n\Gamma_n^{cl}
+c\left(\Upsilon_n^{lr}\right)^3\\
\end{cases}
\end{eqnarray}
with the initial conditions
$\Gamma_1^{cl}(1,1,c)=\Upsilon_1^{lr}(1,1,c)
=\Upsilon_1^{lu}(1,1,c)=\Upsilon_1^{ru}(1,1,c)=1+c$.

One can prove, by induction on $n$, that the solutions of the
system \eqref{easysystemeasy} are
$$
\Gamma_n^{cl}(1,1,c)=\Upsilon_n^{lr}(1,1,c)=\Upsilon_n^{lu}(1,1,c)=\Upsilon_n^{ru}(1,1,c)=(1+c)^{\frac{3^n-1}{2}}
\ \ \mbox{ for each }n.
$$
By studying the derivatives of the function
$\log(\Gamma_n^{cl}(1,1,c))$ with respect to $c$, one gets:
$$
\mu_{n,c}=\frac{3^n-1}{4} \qquad \sigma_{n,c}^2=\frac{3^n-1}{8}.
$$
Symmetry of the labeling of the graph ensures that  the same
values arise for the random variables $a_n,b_n$.
\end{proof}

\subsection{The Sierpi\'{n}ski graphs}

The Sierpi\'nski graphs $\Omega_n$ being not regular, they cannot
be realized as Schreier graphs of any group. There exist however a
number of natural, geometric labelings of edges of $\Omega_n$ by
letters $a,b,c$ (see \cite{noi1}). Here we will be interested in
one particular labeling that is obtained by considering the
labeled Schreier graph $\Sigma_n$ of the Hanoi Towers group and
then by contracting the edges connecting copies of $\Sigma_{n-1}$
in $\Sigma_n$; and so we call this the "Schreier" labeling of
$\Omega_n$.
\begin{os}\label{geomlabel}\rm The "Schreier" labeling on $\Omega_n$ can be constructed recursively, as follows.
Start with the graph $\Omega_1$ in the picture below; then, for
each $n\geq 2$, the graph $\Omega_n$ is defined as the union of
three copies of $\Omega_{n-1}$. For each one of the out-most
(corner) vertices of $\Omega_n$, the corresponding copy of
$\Omega_{n-1}$ is reflected with respect to the bisector of the
corresponding angle.
\end{os}
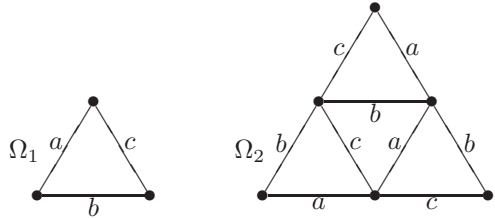
\unitlength=0,25mm
\begin{center}
\begin{picture}(400,110)
\put(85,30){$\Omega_1$}\put(205,30){$\Omega_2$}

\letvertex A=(130,60)\letvertex B=(100,10)\letvertex C=(160,10)

\letvertex D=(280,110)\letvertex E=(250,60)\letvertex F=(220,10)\letvertex G=(280,10)

\letvertex H=(340,10)\letvertex I=(310,60)

\drawvertex(A){$\bullet$}\drawvertex(B){$\bullet$}

\drawvertex(C){$\bullet$}\drawvertex(D){$\bullet$}

\drawvertex(E){$\bullet$}\drawvertex(F){$\bullet$}

\drawvertex(G){$\bullet$}\drawvertex(H){$\bullet$}

\drawvertex(I){$\bullet$}

\drawundirectededge(B,A){$a$} \drawundirectededge(C,B){$b$}

\drawundirectededge(A,C){$c$} \drawundirectededge(E,D){$c$}

\drawundirectededge(F,E){$b$} \drawundirectededge(G,F){$a$}

\drawundirectededge(H,G){$c$} \drawundirectededge(I,H){$b$}

\drawundirectededge(D,I){$a$} \drawundirectededge(I,E){$b$}

\drawundirectededge(E,G){$c$} \drawundirectededge(G,I){$a$}
\end{picture}
\end{center}
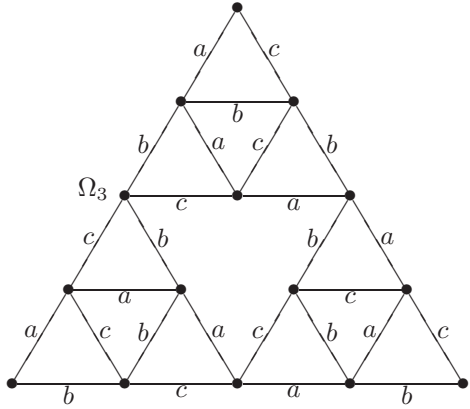
\begin{center}
\begin{picture}(400,205)
\put(135,110){$\Omega_3$}

\letvertex A=(220,210)\letvertex B=(190,160)\letvertex C=(160,110)

\letvertex D=(130,60)\letvertex E=(100,10)\letvertex F=(160,10)\letvertex G=(220,10)

\letvertex H=(280,10)\letvertex I=(340,10)

\letvertex L=(310,60)\letvertex M=(280,110)\letvertex N=(250,160)

\letvertex O=(220,110)\letvertex P=(190,60)\letvertex Q=(250,60)

\drawvertex(A){$\bullet$}\drawvertex(B){$\bullet$}

\drawvertex(C){$\bullet$}\drawvertex(D){$\bullet$}

\drawvertex(E){$\bullet$}\drawvertex(F){$\bullet$}

\drawvertex(G){$\bullet$}\drawvertex(H){$\bullet$}

\drawvertex(I){$\bullet$}\drawvertex(L){$\bullet$}\drawvertex(M){$\bullet$}

\drawvertex(N){$\bullet$}\drawvertex(O){$\bullet$}

\drawvertex(P){$\bullet$}\drawvertex(Q){$\bullet$}

\drawundirectededge(E,D){$a$} \drawundirectededge(D,C){$c$}

\drawundirectededge(C,B){$b$} \drawundirectededge(B,A){$a$}

\drawundirectededge(A,N){$c$} \drawundirectededge(N,M){$b$}

\drawundirectededge(M,L){$a$} \drawundirectededge(L,I){$c$}

\drawundirectededge(I,H){$b$} \drawundirectededge(H,G){$a$}

\drawundirectededge(G,F){$c$} \drawundirectededge(F,E){$b$}

\drawundirectededge(N,B){$b$} \drawundirectededge(O,C){$c$}

\drawundirectededge(M,O){$a$} \drawundirectededge(P,D){$a$}

\drawundirectededge(L,Q){$c$} \drawundirectededge(B,O){$a$}

\drawundirectededge(O,N){$c$} \drawundirectededge(C,P){$b$}

\drawundirectededge(P,G){$a$} \drawundirectededge(D,F){$c$}

\drawundirectededge(Q,M){$b$} \drawundirectededge(G,Q){$c$}

\drawundirectededge(H,L){$a$}\drawundirectededge(F,P){$b$}

\drawundirectededge(Q,H){$b$}
\end{picture}
\end{center}
Let $\Upsilon_n^{lr}(a,b,c)$, $\Upsilon_n^{lu}(a,b,c)$ and
$\Upsilon_n^{ru}(a,b,c)$ be defined as for the Schreier graphs
$\Sigma_n$ of the Hanoi Towers group in the previous subsection.
Then one can easily check that these functions satisfy the
following system of equations:
$$
\begin{cases}
\Gamma_{n+1}^{cl} = \left(\Gamma_n^{cl}\right)^3+\Upsilon_n^{lr}\Upsilon_n^{lu}\Upsilon_n^{ru}\\
\Upsilon_{n+1}^{lu} =\Upsilon_n^{lr}\Upsilon^{ru}_n\Gamma_n^{cl}
+\left(\Upsilon_n^{lu}\right)^3\\
\Upsilon_{n+1}^{ru} =\Upsilon_n^{lu}\Upsilon^{lr}_n\Gamma_n^{cl}
+\left(\Upsilon_n^{ru}\right)^3\\
\Upsilon_{n+1}^{lr} =\Upsilon_n^{lu}\Upsilon^{ru}_n\Gamma_n^{cl}
+\left(\Upsilon_n^{lr}\right)^3\\
\end{cases}
$$
with the initial conditions $\Gamma_1^{cl}(a,b,c)=1+abc$,
$\Upsilon_1^{lr}(a,b,c) = ac+b$, $\Upsilon_1^{lu}(a,b,c)=a+bc$,
$\Upsilon_1^{ru}(a,b,c)=c+ab$. \\ \indent Proceeding as in
Subsection \ref{hanoistatisticsfeb}, we find:
$$
\Gamma_n^{cl}(1,1,c)=\Upsilon_n^{lr}(1,1,c)=\Upsilon_n^{lu}(1,1,c)=\Upsilon_n^{ru}(1,1,c)=2^{\frac{3^{n-1}-1}{2}}(1+c)^{3^{n-1}},
$$
which implies the following

\begin{prop}\label{Schrlabel}
The mean and the variance for the random variable $w_n$, with
$w=a,b,c$, for $\Omega_n$ with the "Schreier" labeling are:
$$
\mu_{n,w}=\frac{3^{n-1}}{2} \qquad
\sigma_{n,w}^2=\frac{3^{n-1}}{4}.
$$
The random variables $a_n, b_n, c_n$ are asymptotically normal.
\end{prop}
It is interesting to compare these computations with those for a
different labeling of $\Omega_n$, that we call the
"rotation-invariant" labeling of Sierpi\'nski graphs, defined
recursively as follows. (Compare the construction to the recursive
description of the "Schreier labeling" in Remark \ref{geomlabel}.)

Let $\Omega_2$ be the weighted graph in the following picture.
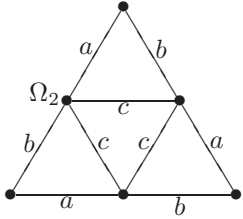
\begin{center}
\begin{picture}(400,110)
\put(155,55){$\Omega_2$}

\letvertex D=(205,105)\letvertex E=(175,55)\letvertex F=(145,5)\letvertex G=(205,5)

\letvertex H=(265,5)\letvertex I=(235,55)

\drawvertex(D){$\bullet$}

\drawvertex(E){$\bullet$}\drawvertex(F){$\bullet$}

\drawvertex(G){$\bullet$}\drawvertex(H){$\bullet$}

\drawvertex(I){$\bullet$}

\drawundirectededge(E,D){$a$} \drawundirectededge(F,E){$b$}

\drawundirectededge(G,F){$a$}

\drawundirectededge(H,G){$b$} \drawundirectededge(I,H){$a$}

\drawundirectededge(D,I){$b$} \drawundirectededge(I,E){$c$}

\drawundirectededge(E,G){$c$} \drawundirectededge(G,I){$c$}
\end{picture}
\end{center}
Then define, for each $n\geq 3$,  $\Omega_n$ as the union of three
copies of $\Omega_{n-1}$, rotated by $k\pi/3$ with $k=0,1,2$.

For $n=3$, one gets the following graph.
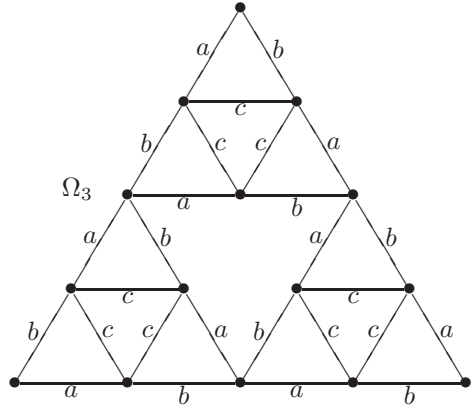
\begin{center}
\begin{picture}(400,210)
\put(110,110){$\Omega_3$}

\letvertex A=(205,210)\letvertex B=(175,160)\letvertex C=(145,110)

\letvertex D=(115,60)\letvertex E=(85,10)\letvertex F=(145,10)\letvertex G=(205,10)

\letvertex H=(265,10)\letvertex I=(325,10)

\letvertex L=(295,60)\letvertex M=(265,110)\letvertex N=(235,160)

\letvertex O=(205,110)\letvertex P=(175,60)\letvertex Q=(235,60)

\drawvertex(A){$\bullet$}\drawvertex(B){$\bullet$}

\drawvertex(C){$\bullet$}\drawvertex(D){$\bullet$}

\drawvertex(E){$\bullet$}\drawvertex(F){$\bullet$}

\drawvertex(G){$\bullet$}\drawvertex(H){$\bullet$}

\drawvertex(I){$\bullet$}\drawvertex(L){$\bullet$}\drawvertex(M){$\bullet$}

\drawvertex(N){$\bullet$}\drawvertex(O){$\bullet$}

\drawvertex(P){$\bullet$}\drawvertex(Q){$\bullet$}

\drawundirectededge(E,D){$b$} \drawundirectededge(D,C){$a$}

\drawundirectededge(C,B){$b$} \drawundirectededge(B,A){$a$}

\drawundirectededge(A,N){$b$} \drawundirectededge(N,M){$a$}

\drawundirectededge(M,L){$b$} \drawundirectededge(L,I){$a$}

\drawundirectededge(I,H){$b$} \drawundirectededge(H,G){$a$}

\drawundirectededge(G,F){$b$} \drawundirectededge(F,E){$a$}

\drawundirectededge(N,B){$c$} \drawundirectededge(O,C){$a$}

\drawundirectededge(M,O){$b$} \drawundirectededge(P,D){$c$}

\drawundirectededge(L,Q){$c$} \drawundirectededge(B,O){$c$}

\drawundirectededge(O,N){$c$} \drawundirectededge(C,P){$b$}

\drawundirectededge(P,G){$a$} \drawundirectededge(D,F){$c$}

\drawundirectededge(Q,M){$a$} \drawundirectededge(G,Q){$b$}

\drawundirectededge(H,L){$c$}\drawundirectededge(F,P){$c$}

\drawundirectededge(Q,H){$c$}
\end{picture}
\end{center}
It turns out that the weighted generating function of closed
polygons is easier to compute for $\Omega_n$ with the
"rotation-invariant" labeling, than with the "Schreier" labeling.
More precisely, we have the following
\begin{teo}\label{ultimoteocitato}
For each $n\geq 2$, the weighted generating function of closed
polygons for the graph $\Omega_n$ with the "rotation-invariant"
labeling is
$$
\Gamma_n^{cl}(a,b,c)=((a+bc)(b+ac))^\frac{7\cdot
3^{n-2}}{4}\psi_1^{3^{n-2}}\!(a,b,c)\prod_{k=2}^n\psi_k^{3^{n-k}}\!(a,b,c)\cdot
(\psi_{n+1}(a,b,c)-1)
$$
where $\psi_1(a,b,c) = \frac{1+c}{((a+bc)(b+ac))^{\frac{1}{4}}}$,
$\psi_2(a,b,c)=\frac{1+ab}{((a+bc)(b+ac))^{\frac{1}{2}}}$,\\
$\psi_3(a,b,c)=\frac{a^2b^2c^2-a^2b^2c+a^2b^2+4abc+c^2-c+1+a^2c+b^2c}{(a+bc)(b+ac)}$
and, for each $k\geq 4$,
$$
\psi_k(a,b,c)=\psi_{k-1}^2(a,b,c)-3\psi_{k-1}(a,b,c)+4.
$$
\end{teo}
\begin{proof}
Consider the graph $\Omega_n$. For each $n\geq 2$, define the sets
$P_n$ and $L_n$ as in Subsection \ref{grafohanoigrafo} and let
$\Gamma_n^{cl}(a,b,c)$ and $\Upsilon_n(a,b,c)$ be the associated
weighted generating functions. By using the symmetry of the
labeling, one can check that these functions satisfy the following
equations
\begin{eqnarray}\label{sistgenetiq}
\begin{cases}
\Gamma_n^{cl}(a,b,c) = \left(\Gamma_{n-1}^{cl}(a,b,c)\right)^3+\Upsilon_{n-1}^3(a,b,c)\\
\Upsilon_n(a,b,c) =
\Upsilon_{n-1}^3(a,b,c)+\Upsilon_{n-1}^2(a,b,c)\Gamma_{n-1}^{cl}(a,b,c).
\end{cases}
\end{eqnarray}
with the initial conditions
$$
\begin{cases}
\Gamma_2^{cl}=(1+c)(1+ab)(a^2b^2c^2-a^2b^2c+a^2b^2+4abc-abc^2-ab+c^2-c+1)\\
\Upsilon_2= (1+c)(1+ab)(a+bc)(b+ac).
\end{cases}
$$
As in the proof of Theorem \ref{teosierpinskigenfunct}, one shows
by induction on $n$ that the solutions of the system are
$$
\begin{cases}
\Gamma_n^{cl}=((a+bc)(b+ac))^\frac{7\cdot
3^{n-2}}{4}\psi_1^{3^{n-2}}\!(a,b,c)\prod_{k=2}^n\psi_k^{3^{n-k}}\!(a,b,c)\!\cdot\!
(\psi_{n+1}\!(a,b,c)-1)\\
\Upsilon_n=((a+bc)(b+ac))^\frac{7\cdot
3^{n-2}}{4}\psi_1^{3^{n-2}}(a,b,c)\prod_{k=2}^n\psi_k^{3^{n-k}}(a,b,c)
\end{cases}
$$
\end{proof}
\begin{os}\rm
Although the labels $a$ and $b$ are not symmetric to the label $c$
in the "rotation-invariant" labeling,  computations show that the
functions $\Gamma_n^{cl}(a,1,1)$, $\Gamma_n^{cl}(1,b,1)$ and
$\Gamma_n^{cl}(1,1,c)$ are the same in this case as in the case of
the "Schreier" labeling. It follows that the values of the mean
and the variance of the random variables $a_n, b_n, c_n$ remain
the same as in the "Schreier" labeling, see Proposition
\ref{Schrlabel}.
\end{os}
\subsection{Correspondences via Fisher's Theorem}\label{fisherfisher}

In \cite{fisher} M. Fisher proposed a method of computation for
the partition function of the Ising model on a (finite) planar
lattice $Y$ by relating it to the partition function of the dimers
model (with certain weights) on another planar lattice $Y^\Delta$
constructed from $Y$. (The latter partition function can then be
found by computing the corresponding Pfaffian given by Kasteleyn's
theorem.) This method uses the expression \eqref{polygons} for the
partition function in terms of the generating function of closed
polygons in $Y$. The new lattice $Y^{\Delta}$ is constructed in
such a way that Ising polygon configurations on $Y$ are in
one-to-one correspondence with dimer configurations on
$Y^{\Delta}$. In order to have equality of generating functions
however, the edges of $Y^{\Delta}$ should be weighted in such a
way that the edges coming from $Y$ have the same weight
$\tanh(\beta J_{i,j})$ as in the RHS of \eqref{polygons}, and
other edges have weight 1.

Applying Fisher's construction to Sierpi\'nski graphs, one
concludes easily that if $Y=\Omega_n$ for some $n\geq 1$, then
$Y^\Delta = \tilde\Sigma_{n+1}$, the $(n+1)$-st Schreier graph of
the Hanoi Towers group $H^{(3)}$ with three corner vertices
deleted. Note that the corner vertices are the only vertices in
$\Sigma_n$ with loops attached to them, and so it is anyway
natural to forget about them when counting dimer coverings. The
construction consists in applying to $Y$ the following
substitutions, where edges labeled by $e$ in $Y^{\Delta}$ are in
bijection with edges in $Y$, and should be assigned weight
$\tanh(\beta J_{i,j})$. Other edges should be assigned weight $1$.
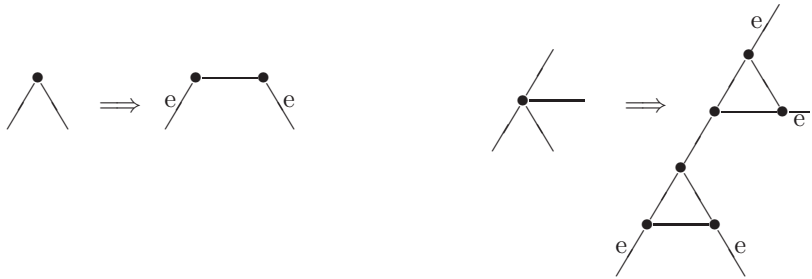
\unitlength=0,3mm
\begin{center}
\begin{picture}(400,130)
\letvertex A=(20,90)\letvertex B=(5,65)
\letvertex C=(35,65)\letvertex D=(75,65)
\letvertex E=(90,90)\letvertex F=(120,90)
\letvertex G=(135,65)

\letvertex H=(250,105)\letvertex I=(235,80)

\letvertex L=(220,55)\letvertex M=(265,80)

\letvertex N=(250,55)

\letvertex O=(350,125)\letvertex P=(335,100)

\letvertex Q=(320,75)\letvertex R=(305,50)\letvertex S=(290,25)\letvertex T=(275,0)

\letvertex U=(350,75)\letvertex V=(365,75)\letvertex Z=(320,25)\letvertex X=(335,0)

\put(47,75){$\Longrightarrow$}\put(280,75){$\Longrightarrow$}

\drawvertex(A){$\bullet$}

\drawvertex(E){$\bullet$}\drawvertex(F){$\bullet$}

\drawvertex(I){$\bullet$}

\drawvertex(P){$\bullet$}\drawvertex(Q){$\bullet$}

\drawvertex(U){$\bullet$}\drawvertex(R){$\bullet$}

\drawvertex(S){$\bullet$}\drawvertex(Z){$\bullet$}

\drawundirectededge(A,B){}\drawundirectededge(A,C){}\drawundirectededge(D,E){e}

\drawundirectededge(F,E){}\drawundirectededge(F,G){e}\drawundirectededge(H,I){}

\drawundirectededge(I,L){}\drawundirectededge(M,I){}\drawundirectededge(I,N){}

\drawundirectededge(P,O){e}\drawundirectededge(Q,P){}\drawundirectededge(P,U){}

\drawundirectededge(V,U){e}\drawundirectededge(R,Q){}\drawundirectededge(U,Q){}

\drawundirectededge(S,R){}\drawundirectededge(R,Z){}\drawundirectededge(Z,S){}

\drawundirectededge(T,S){e}\drawundirectededge(Z,X){e}
\end{picture}
\end{center}
The correspondence between closed polygons in $Y$ and dimer
coverings of $Y^\Delta$ is as follows: if an edge in $Y$ belongs
to a closed polygon, then the corresponding $e$-edge in $Y^\Delta$
does not belong to the dimer covering of $Y^\Delta$ associated
with that closed polygon, and vice versa.

The following pictures give an example of a closed polygon in
$\Omega_2$ and of the associated dimer covering of
$\tilde\Sigma_3$: \unitlength=0,25mm
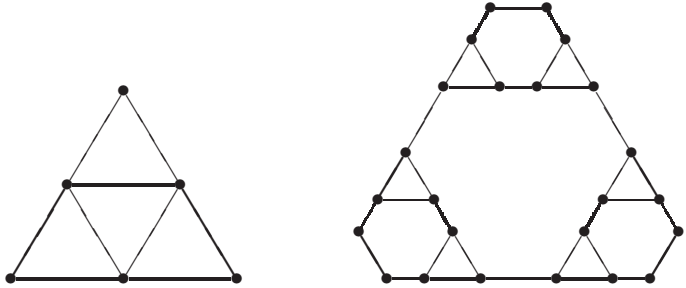
\begin{center}
\begin{picture}(400,190)
\letvertex D=(60,110)\letvertex E=(30,60)\letvertex F=(0,10)\letvertex G=(60,10)

\letvertex H=(120,10)\letvertex I=(90,60)

\thinlines \drawvertex(D){$\bullet$}

\drawvertex(E){$\bullet$}\drawvertex(F){$\bullet$}

\drawvertex(G){$\bullet$}\drawvertex(H){$\bullet$}

\drawvertex(I){$\bullet$}

\drawundirectededge(E,D){} \drawundirectededge(D,I){}

 \drawundirectededge(E,G){}

\drawundirectededge(G,I){}

\thicklines \drawundirectededge(I,E){}\drawundirectededge(F,E){}

\drawundirectededge(G,F){}

\drawundirectededge(H,G){}\drawundirectededge(I,H){}
\letvertex B=(255,154)

\letvertex C=(245,137)\letvertex D=(230,112)

\letvertex E=(210,77)\letvertex F=(195,52)

\letvertex G=(185,35)\letvertex I=(200,10)

\letvertex L=(220,10)\letvertex M=(250,10)

\letvertex N=(290,10)

\letvertex O=(320,10)\letvertex P=(340,10)\letvertex R=(355,35)

\letvertex S=(345,52)\letvertex T=(330,77)

\letvertex U=(310,112)\letvertex V=(295,137)\letvertex Z=(285,154)

\letvertex J=(260,112)\letvertex K=(280,112)

\letvertex W=(225,52)

\letvertex Y=(235,35)\letvertex X=(315,52)

\letvertex d=(305,35)

\drawvertex(B){$\bullet$}

\drawvertex(C){$\bullet$}\drawvertex(D){$\bullet$}

\drawvertex(E){$\bullet$}\drawvertex(F){$\bullet$}

\drawvertex(G){$\bullet$} \drawvertex(I){$\bullet$}

\drawvertex(L){$\bullet$}\drawvertex(M){$\bullet$}

\drawvertex(N){$\bullet$}\drawvertex(O){$\bullet$}

\drawvertex(P){$\bullet$}

\drawvertex(R){$\bullet$}\drawvertex(S){$\bullet$}

\drawvertex(T){$\bullet$}

\drawvertex(U){$\bullet$}\drawvertex(V){$\bullet$}

\drawvertex(Z){$\bullet$}\drawvertex(W){$\bullet$}

\drawvertex(J){$\bullet$}\drawvertex(K){$\bullet$}

\drawvertex(X){$\bullet$}\drawvertex(Y){$\bullet$}

\drawvertex(d){$\bullet$}

\thinlines  \drawundirectededge(C,J){} \drawundirectededge(E,D){}

\drawundirectededge(W,F){}

\drawundirectededge(L,Y){}\drawundirectededge(N,M){}

\drawundirectededge(d,O){} \drawundirectededge(R,P){}

\drawundirectededge(S,X){} \drawundirectededge(U,T){}

\drawundirectededge(K,V){}\drawundirectededge(D,C){}\drawundirectededge(X,T){}\drawundirectededge(S,R){}

\drawundirectededge(P,O){}\drawundirectededge(d,N){}\drawundirectededge(Y,M){}\drawundirectededge(G,F){}\drawundirectededge(E,W){}

\drawundirectededge(V,U){}\drawundirectededge(I,L){}\drawundirectededge(B,Z){}\drawundirectededge(J,K){}

\thicklines \drawundirectededge(B,C){}\drawundirectededge(V,Z){}

\drawundirectededge(J,D){}\drawundirectededge(U,K){}\drawundirectededge(F,E){}\drawundirectededge(M,L){}

\drawundirectededge(Y,W){}\drawundirectededge(T,S){}\drawundirectededge(X,d){}\drawundirectededge(O,N){}\drawundirectededge(R,P){}

\drawundirectededge(G,I){}
\end{picture}
\end{center}
If, for a certain $n\geq 1$, the Sierpi\'nski graph $\Omega_n$ is
considered with the "Schreier" labeling, then the labeling of the
graph $\tilde\Sigma_{n+1}$ given by Fisher's construction will be
a restriction of the usual Schreier labeling of $\Sigma_{n+1}$.
More precisely, only the edges that connect copies of
$\Sigma_{n-1}$ but not copies of $\Sigma_n$ will be labeled (other
edges have weight $1$), and the labels are the same as in the
standard labeling of $\Sigma_{n+1}$ as a Schreier graph of the
group $H^{(3)}$. The following picture represents $\tilde\Sigma_3$
as $Y^\Delta$ with $Y=\Omega_2$ with the "Schreier" labeling.
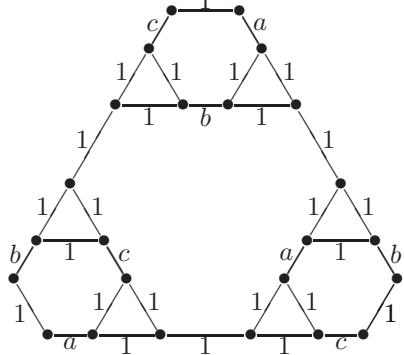
\unitlength=0,3mm
\begin{center}
\begin{picture}(400,160)
\letvertex B=(160,154)

\letvertex C=(150,137)\letvertex D=(135,112)

\letvertex E=(115,77)\letvertex F=(100,52)

\letvertex G=(90,35)\letvertex I=(105,10)

\letvertex L=(125,10)\letvertex M=(155,10)

\letvertex N=(195,10)

\letvertex O=(225,10)\letvertex P=(245,10)\letvertex R=(260,35)

\letvertex S=(250,52)\letvertex T=(235,77)

\letvertex U=(215,112)\letvertex V=(200,137)\letvertex Z=(190,154)

\letvertex J=(165,112)\letvertex K=(185,112)

\letvertex W=(130,52)

\letvertex Y=(140,35)\letvertex X=(220,52)

\letvertex d=(210,35)

\drawvertex(B){$\bullet$}

\drawvertex(C){$\bullet$}\drawvertex(D){$\bullet$}

\drawvertex(E){$\bullet$}\drawvertex(F){$\bullet$}

\drawvertex(G){$\bullet$} \drawvertex(I){$\bullet$}

\drawvertex(L){$\bullet$}\drawvertex(M){$\bullet$}

\drawvertex(N){$\bullet$}\drawvertex(O){$\bullet$}

\drawvertex(P){$\bullet$}

\drawvertex(R){$\bullet$}\drawvertex(S){$\bullet$}

\drawvertex(T){$\bullet$}

\drawvertex(U){$\bullet$}\drawvertex(V){$\bullet$}

\drawvertex(Z){$\bullet$}\drawvertex(W){$\bullet$}

\drawvertex(J){$\bullet$}\drawvertex(K){$\bullet$}

\drawvertex(X){$\bullet$}\drawvertex(Y){$\bullet$}

\drawvertex(d){$\bullet$}

\drawundirectededge(C,J){1} \drawundirectededge(E,D){1}

\drawundirectededge(W,F){1}

\drawundirectededge(L,Y){1}\drawundirectededge(N,M){1}

\drawundirectededge(d,O){1} \drawundirectededge(R,P){1}

\drawundirectededge(S,X){1} \drawundirectededge(U,T){1}

\drawundirectededge(K,V){1}\drawundirectededge(D,C){1}\drawundirectededge(X,T){1}\drawundirectededge(S,R){$b$}

\drawundirectededge(P,O){$c$}\drawundirectededge(N,d){1}\drawundirectededge(Y,M){1}\drawundirectededge(G,F){$b$}\drawundirectededge(E,W){1}

\drawundirectededge(V,U){1}\drawundirectededge(L,I){$a$}\drawundirectededge(B,Z){1}\drawundirectededge(K,J){$b$}

\drawundirectededge(C,B){$c$}\drawundirectededge(Z,V){$a$}

\drawundirectededge(J,D){1}\drawundirectededge(U,K){1}\drawundirectededge(F,E){1}\drawundirectededge(M,L){1}

\drawundirectededge(W,Y){$c$}\drawundirectededge(T,S){1}\drawundirectededge(d,X){$a$}\drawundirectededge(O,N){1}\drawundirectededge(R,P){1}

\drawundirectededge(I,G){1}
\end{picture}
\end{center}
\begin{os}\rm One can wonder what Fisher's construction gives for $\{\Sigma_n\}_{n\geq 1}$. It turns out that
if $Y=\Sigma_n$, the $n$-th Schreier graph of $H^{(3)}$, then
$Y^\Delta=\tilde\Sigma_{n+1}$, the same as for $Y=\Omega_n$, the
$n$-th Sierpi\'nski graph .
\end{os}


\begin{thebibliography}{99}
\bibitem{bondarenkothesis} I. Bondarenko, Groups generated by
bounded automata and their Schreier graphs, PhD Thesis Texas A\&M,
2007, available
at \texttt{http://\\
txspace.tamu.edu/bitstream/handle/1969.1/85845/Bondarenko.pdf?sequence=1}

\bibitem{burioni} R. Burioni, D. Cassi and L. Donetti, Lee-Yang
zeros and the Ising model on the Sierpi\'{n}ski Gasket, {\it J.
Phys. A: Math. Gen.}, {\bf 32} (1999), 5017--5027.


\bibitem{noi1} D. D'Angeli, A. Donno and T. Nagnibeda, The dimer model on some families of self-similar
graphs, preprint.


\bibitem{schreierbasilica} D. D'Angeli, A. Donno, M. Matter and T. Nagnibeda, Schreier graphs of the Basilica group,
submitted, available at \texttt{http://arxiv.org/abs/0911.2915}

\bibitem{fisher} M. E. Fisher, On the dimer solution of planar
Ising models, {\it J. Math. Phys.}, {\bf 7} (1966), no. 10,
1776--1781.

\bibitem{Mand} Y. Gefen, A. Aharony, Y. Shapir and B. Mandelbrot, Phase transitions on fractals. II. Sierpi\'nski gaskets, {\it J. Phys. A}, {\bf 17} (1984), no. 2,  435--444.



\bibitem{grigorchuk} R. I. Grigorchuk, Solved and unsolved problems around one group, in: ``Infinite groups:
geometric, combinatorial and dynamical aspects" (L. Bartholdi, T.
Ceccherini-Silberstein, T. Smirnova-Nagnibeda and A. \.{Z}uk
editors), Progr. Math., {\bf 248}, Birkh\"auser, Basel, 2005,
117--218.

\bibitem{hanoi} R. I. Grigorchuk and Z. \v{S}uni\'{c},
Self-similarity and branching in group theory, in: ``Groups St.
Andrews 2005, I", London Math. Soc. Lecture Note Ser., {\bf 339},
Cambridge Univ. Press, Cambridge, 2007, 36--95.


\bibitem{Grom} M. Gromov, Structures m\'etriques pour les vari\'et\'es riemanniennes, Edited by J. Lafontaine and P. Pansu, Textes Math\'ematiques,
1. CEDIC, Paris, 1981.



\bibitem{ising} E. Ising, Beitrag zur Theorie des
Ferromagnetismus, {\it Zeit. f\"{u}r Physik}, {\bf 31} (1925),
253--258.

\bibitem{lenz} W. Lenz, Beitr\"age zum Verst\"andnis der magnetischen Eigenschaften in festen K\"orpern, {\it Physikalische Zeitschrift}, {\bf 21} (1920), 613--615.

\bibitem{volo} V. Nekrashevych, {\it Self-similar Groups}, Mathematical Surveys and
Monographs, 117. American Mathematical Society, Providence, RI,
2005.

\bibitem{wagner2} E. Teufl and S. Wagner, Enumeration of matchings in families of self-similar
graphs, submitted.
\end{thebibliography}
\end{document}